\documentclass[a4paper,10pt,reqno]{elsarticle}
\usepackage{amsmath}
\usepackage{cases}

\UseRawInputEncoding

\usepackage{amsfonts}
\usepackage[colorlinks,linkcolor=blue,citecolor=blue]{hyperref}
\usepackage{latexsym, amssymb, amsmath, amsthm, bbm}
\usepackage[all]{xy}
\usepackage{pgfplots}
\usepackage{enumerate}

\DeclareSymbolFont{EulerExtension}{U}{euex}{m}{n}
\DeclareMathSymbol{\euintop}{\mathop} {EulerExtension}{"52}
\DeclareMathSymbol{\euointop}{\mathop} {EulerExtension}{"48}

\allowdisplaybreaks[4]

\def \id{\operatorname{id}}

\def \C{\mathcal{C}}

\def \N{\mathbb{N}}
\def \Z{\mathbb{Z}}

\def \k{\Bbbk}

\def \C{\mathcal{C}}
\def \D{\Delta}

\def \N{\mathbb{N}}

\def \A{\mathcal{A}}
\def \B{\mathcal{B}}
\def \C{\mathcal{C}}
\def \D{\mathcal{D}}
\def \E{\mathcal{E}}
\def \F{\mathcal{F}}
\def \G{\mathcal{G}}
\def \X{\mathcal{X}}
\def \Y{\mathcal{Y}}
\def \Z{\mathcal{Z}}

\def \la{\lambda}
\def \om{\omega}

\usepackage{mathtools}

\def\joinrel{\mathrel{\mkern-4mu}}
\newcommand{\blackltimes}{\mathop{\raisebox{0.2ex}{${\scriptstyle \blacktriangleright\joinrel<}$}}}
\newcommand{\dotltimes}
{\mathop{\raisebox{0.12ex}{$\shortmid$}\raisebox{0.2ex}{\makebox[0.86em][r]{${\scriptstyle\gtrdot\joinrel<}$}}}}

\numberwithin{equation}{section}

\newtheorem{theorem}{Theorem}[section]
\newtheorem{lemma}[theorem]{Lemma}
\newtheorem{proposition}[theorem]{Proposition}
\newtheorem{corollary}[theorem]{Corollary}
\newtheorem{definition}[theorem]{Definition}
\newtheorem{example}[theorem]{Example}
\newtheorem{remark}[theorem]{Remark}
\newtheorem{question}[theorem]{Question}

\newtheorem{notation}[theorem]{Notation}

\begin{document}
\title{The Link-Indecomposable Components of Hopf Algebras and Their Products}

\author{Kangqiao Li}
\ead{kqli@nju.edu.cn}
\address{Department of Mathematics, Nanjing University, Nanjing 210093, China}

\date{}

\begin{abstract}
The link relation on simple subcoalgebras is used for decompositions of coalgebras. In this paper, we provide more sufficient conditions for this link relation, and prove a formula on the products between link-indecomposable components of Hopf algebras with the dual Chevalley property.
Furthermore, we show that each of its component is generated by a simple subcoalgebra,
as a faithfully flat module (in fact, a projective generator) over a Hopf subalgebra which is the component containing the unit element.
Our conclusions generalize some relevant results on pointed Hopf algebras, which were established by Montgomery in 1995.
\end{abstract}


\maketitle

\section{Introduction}

It is known in Kaplansky \cite{Kap75} that any coalgebra could be written uniquely as a direct sum of indecomposable subcoalgebras. The notion of the link relation (also known as the connected relation) on simple subcoalgebras is a theoretical way to determine the direct summands, which are referred as the link-indecomposable components. This was firstly shown by Shudo and Miyamoto \cite{SM78}. Later in 1995, Montgomery \cite{Mon95} refined the related knowledge with the language of quivers, and studied properties of the link-indecomposable components of a pointed Hopf algebra.
For any pointed Hopf algebra $H$, she established a formula on the products of link-indecomposable components. As consequences, $H$ is free over a normal Hopf subalgebra $H_{(1)}$ which is exactly the link-indecomposable component containing the unit element, and $H$ is furthermore a crossed product of a group over $H_{(1)}$.

This paper is devoted to generalize some of these main results in \cite{Mon95} to non-pointed Hopf algebras. Denote the link-indecomposable component of $H$ containing the simple subcoalgebra $E$ by $H_{(E)}$, which is a subcoalgebra of $H$, and note again that $H$ is the direct sum of its different link-indecomposable components. Our final result is Theorem \ref{thm:DCPnew}, stating that:
\begin{theorem}\label{thm.1}
Let $H$ be a Hopf algebra over an arbitrary field $\k$ with the dual Chevalley property. Denote the set of all the simple subcoalgebras of $H$ by $\mathcal{S}$. Then
\begin{itemize}
  \item[(1)] For any $C\in\mathcal{S}$, $H_{(C)}=CH_{(1)}=H_{(1)}C$;
  \item[(2)] For any $C,D\in\mathcal{S}$,
    $H_{(C)}H_{(D)}\subseteq\sum\limits_{E\in\mathcal{S},\;E\subseteq CD}H_{(E)}$;
  \item[(3)] $H_{(1)}$ is a Hopf subalgebra.
\end{itemize}
\end{theorem}

Here by the dual Chevalley property we mean that the coradical $H_0$ is a Hopf subalgebra. It would here deserve a remark that the dual Chevalley property must be assumed to prove the Items (1) and (2) above, as we will show in Subsection \ref{subsection:D} by an explicit example, $D(2,2,\sqrt{-1})$. In addition, the following proposition (=Proposition \ref{prop:H(1)sub}) which pays an essential role for proving the theorem above is shown under a weaker assumption than the dual Chevalley property (with $\k$ assumed to be algebraically closed).
\begin{proposition}
Let $H$ be a Hopf algebra over an algebraically closed field $\k$ with the bijective antipode $S$. If ${(H_{(1)})_0}^3\subseteq H_0$ holds, then $H_{(1)}$ is a Hopf subalgebra.
\end{proposition}

These results for a Hopf algebra $H$ with the dual Chevalley property are proved mainly by two steps. Firstly, we aim to obtain Proposition \ref{thm:DCP}, providing that if the base field $\k$ is algebraically closed, then the claim of Theorem \ref{thm.1}(2) holds. This is shown by a method of non-trivial primitive matrices, which are non-pointed analogues of non-trivial primitive elements. Specifically, some sufficient conditions for simple subcoalgebras to be linked are described by non-trivial matrices, which would help us to study the link relation and link-indecomposable components by straightforward computations on matrices. Then the desired results are possible to be obtained.

As the second step we present the link-indecomposable decomposition of $H$ by applying its description as smash coproduct. In fact, the quotient right $H$-module coalgebra $Q:=H/{H_0}^+H$ of $H$ is made into a right $H_0$-comodule coalgebra, so that $H$ is identified with the associated smash coproduct $H_0\blackltimes Q$, which has $H_0\otimes Q$ as the underlying left $H_0$-module. Let $K$ be the smallest Hopf subalgebra of $H_0$ that coacts on $Q$. The link-indecomposable decomposition is presented as
$$H=\bigoplus_i K_i\otimes Q,$$
where $K_i$ are the simple subobjects of $H_0$ in the semisimple category ${}^{H_0}\!\mathcal{M}^{H_0}_K$ of Hopf bicomodules. This presentation essentially coincides with Montgomery's, when $H$ is pointed, in particular; see the paragraph following Corollary \ref{cor:Hdecomp}.

Moreover as a direct corollary, the faithful flatness of $H$ over the Hopf subalgebra $H_{(1)}$ is followed when $H$ has the dual Chevalley property, which also implies that $H$ is a projective generator as an $H_{(1)}$-module. This appears as Corollary \ref{cor:ffness}(1) in the paper:
\begin{corollary}
Let $H$ be a Hopf algebra with the dual Chevalley property. Then $H$ is a projective generator of left as well as right $H_{(1)}$-modules.
\end{corollary}

However, $H$ is not always free as a left or right $H_{(1)}$-module, which is shown with an example in Subsection \ref{subsection:HoverH(1)}. Meanwhile, even though there are examples, such as $T_\infty(2,1,-1)^\bullet$ provided in Subsection \ref{subsection:T}, with the normal link-indecomposable component as a Hopf subalgebra, we do not know when $H_{(1)}$ becomes a normal Hopf subalgebra for a non-pointed Hopf algebra $H$.

The organization of this paper is as follows: In Section \ref{section2}, necessary matric techniques including certain properties of multiplicative and primitive matrices are provided. In Section \ref{section3}, we recall and promote the notions related to the link relations, and show the formula on the products of components of $H$ over an algebraically closed field, with the usage of matric conditions established. Our final theorem would be proved by the constructions of smash coproducts in Section \ref{section4}, where properties of the Hopf subalgebra $H_{(1)}$ is also discussed. At last, by a method to determine the link-decompositions, some examples and applications are given in Section \ref{section5}.

\section{Matrices over Coalgebras}\label{section2}

Through out this paper, all vector spaces, coalgebras, bialgebras and Hopf algebras are assumed to be over a field $\k$. The tensor product over $\k$ is denoted simply by $\otimes$. As the main tools in this paper are matrices over vector spaces, an evident lemma should be noted as first:

\begin{lemma}
Let $V$ be a vector space. For any matrix $\A$ over $V$, the followings are equivalent:
\begin{itemize}
  \item[(1)] All the entries of $\A$ are linearly independent;
  \item[(2)] All the entries of $P\A Q$ are linearly independent, for some invertible matrices $P$ and $Q$ over $\k$.
\end{itemize}
\end{lemma}

Moreover, we always say that two square matrices $\A$ and $\B$ over a vector space $V$ are \textit{similar}, if there exists an invertible matrix $L$ over $\k$ such that $\B=L\A L^{-1}$. This is denoted by $\A\sim\B$ for simplicity.

\subsection{Multiplicative Matrices and Their Operations}

The notion of the multiplicative matrices over coalgebras was once introduced in \cite{Man88}. This helps us generalize some results of pointed coalgebras or Hopf algebras to the case of non-pointed ones.
For our purposes, more properties of multiplicative matrices are considered in this subsection. Let us start by recalling notations and definitions.

\begin{notation}
Let $V$ and $W$ be vector spaces.
\begin{itemize}
\item[(1)] For any matrix $\A:=(v_{ij})_{m\times n}$ over $V$ and matrix $\B:=(w_{ij})_{n\times l}$ over $W$, denote the following matrix
    $$\A\;\widetilde{\otimes}\;\B:=\left(\sum\limits_{k=1}^n v_{ik}\otimes w_{kl}\right)_{m\times l};$$
\item[(2)] For any matrix $\A:=(v_{ij})_{m\times n}$ over $V$, denote the following matrix
    $$\A^\mathrm{T}:=(v_{ji})_{n\times m};$$
\item[(3)] For any linear map $f:V\rightarrow W$ and a matrix $\A:=(v_{ij})_{m\times n}$ over $V$, denote the following matrix
    $$f(\A):=\left(f(v_{ij})\right)_{m\times n}.$$
\end{itemize}
\end{notation}

Then multiplicative matrices could be defined simply as follows.

\begin{definition}
Let $(H,\Delta,\varepsilon)$ be a coalgebra over $\k$.
\begin{itemize}
  \item[(1)] A square matrix $\G$ over $H$ is said to be multiplicative, if $\Delta(\G)=\G\;\widetilde{\otimes}\;\G$ and $\varepsilon(\G)=I$ (the identity matrix over $\k$) both hold;
  \item[(2)] A multiplicative matrix $\C$ is said to be basic, if its entries are linearly independent.
\end{itemize}
\end{definition}

Clearly, all the entries of a basic multiplicative matrix $\C$ span a simple subcoalgebra $C$ of $H$. Conversely, when the base field $\k$ is \textit{algebraically closed}, any simple coalgebra $C$ has a basic multiplicative matrix $\C$ whose entries span $C$. Moreover, we could describe the uniqueness for $\C$ as follows:

\begin{lemma}
Let $C$ be a simple coalgebra over $\k$. Suppose that $\mathcal{C}$ is a basic multiplicative matrix of $C$. Then $\D$ is also a basic multiplicative matrix of $C$ if and only if $\D\sim\C$.
\end{lemma}

\begin{proof}
A particular case of Skolem-Noether theorem follows the fact that: Any two matric bases of a finite-dimensional matrix algebra are similar. Our desired lemma would be its dual version.
\end{proof}

\begin{remark}
One could easily verify that matrices similar to multiplicative ones are also multiplicative, even if they are not basic.
\end{remark}

This lemma states that for a simple coalgebra $C$, its basic multiplicative matrix would be unique up to the similarity relation (over $\k$). In fact as for an arbitrary multiplicative matrix, we claim in the followings that it could be ``decomposed'' into basic ones:

\begin{proposition}\label{prop:sim}
Suppose $\G$ is an $n\times n$ multiplicative matrix over a coalgebra $H$. Then
\begin{itemize}
  \item[(1)] There exist basic multiplicative matrices $\C_1,\C_2,\cdots,\C_t$ over $H$, such that
    $$\G\sim\left(\begin{array}{cccc}
      \C_1 & \X_{12} & \cdots & \X_{1t}  \\
      0 & \C_2 & \cdots & \X_{2t}  \\
      \vdots & \vdots & \ddots & \vdots  \\
      0 & 0 & \cdots & \C_t
    \end{array}\right),$$
    where $\X_{ij}$'s are matrices over $H$ for all $1\leq i<j\leq t$;
  \item[(2)] If all the entries of $\G$ belong to the coradical of $H$, then there exist basic multiplicative matrices $\C_1,\C_2,\cdots,\C_t$ over $H$, such that
      $$\G\sim\left(\begin{array}{cccc}
      \C_1 & 0 & \cdots & 0  \\
      0 & \C_2 & \cdots & 0  \\
      \vdots & \vdots & \ddots & \vdots  \\
      0 & 0 & \cdots & \C_t
    \end{array}\right).$$
\end{itemize}
\end{proposition}

\begin{proof}
It is clear that all the entries of $\G$ span a subcoalgebra $G$ of $H$. Define an $n$-dimensional $\k$-vector space $V:=\k v_1\oplus\k v_2\oplus\cdots\oplus\k v_n$, which becomes a right $G$-comodule with structures
$$\rho(v_1,v_2,\cdots,v_n):=(v_1,v_2,\cdots,v_n)\;\widetilde{\otimes}\;\G.$$
\begin{itemize}
\item[(1)]
Evidently, $V$ has at least one simple $G$-subcomodule, denoted by $W$. Suppose that $W$ has a linear basis $\{w_1,w_2,\cdots,w_r\}$, and
$$\rho(w_1,w_2,\cdots,w_r)=(w_1,w_2,\cdots,w_r)\;\widetilde{\otimes}\;
  \left(\begin{array}{cccc}
      c_{11} & c_{12} & \cdots & c_{1r}  \\  c_{21} & c_{22} & \cdots & c_{2r}  \\
      \vdots & \vdots & \ddots & \vdots  \\  c_{r1} & c_{r2} & \cdots & c_{rr}
  \end{array}\right)$$
holds for some $c_{ij}$'s in $G$. Then according to \cite[Theorem 3.2.11(d)]{Rad12} and its proof, $\{c_{ij}\mid 1\leq i,j\leq r\}$ is linearly independent, and thus spans a simple subcoalgebra with a basic multiplicative matrix $\C_1:=(c_{ij})_{r\times r}$.

Now we suppose $\{w_1,w_2,\cdots,w_r,u_1,u_2\cdots,u_{n-r}\}$ is another linear basis of $V$, which is extended from the basis of $W$ mentioned above. Choose the $n\times n$ transition matrix $L_1$ over $\k$ such that
$$(v_1,v_2,\cdots,v_n)=(w_1,\cdots,w_r,u_1,\cdots,u_{n-r})L_1,$$
and consider the comodule structure $\rho$ at this equation. We could compute to know that
\begin{eqnarray*}
(w_1,\cdots,w_r,u_1,\cdots,u_{n-r})\;\widetilde{\otimes}\;L_1\G
&=&  (w_1,\cdots,w_r,u_1,\cdots,u_{n-r})L_1\;\widetilde{\otimes}\;\G  \\
&=& (v_1,v_2,\cdots,v_n)\;\widetilde{\otimes}\;\G
\;=\; \rho(v_1,v_2,\cdots,v_n)  \\
&=& \rho((w_1,\cdots,w_r,u_1,\cdots,u_{n-r})L_1)  \\
&=& \rho(w_1,\cdots,w_r,u_1,\cdots,u_{n-r})L_1  \\
&=& (w_1,\cdots,w_r,u_1,\cdots,u_{n-r})\;\widetilde{\otimes}
  \left(\begin{array}{cc}  \C_1 & \X_1  \\  0 & \G_1  \end{array}\right)L_1,
\end{eqnarray*}
where $\G_1$ is multiplicative (of size $n-r$) due to the axiom of comodules, and $\X_{1}$ is an $r\times(n-r)$ matrix over $H$. This follows that
$$L_1\G L_1^{-1}=\left(\begin{array}{cc}  \C_1 & \X_1  \\  0 & \G_1  \end{array}\right).$$

If we repeat the process on $\G_1$ for several times, an invertible matrix $L$ over $\k$ could be obtained, such that
$$L\G L^{-1}=\left(\begin{array}{cccc}
      \C_1 & \X_{12} & \cdots & \X_{1t}  \\
      0 & \C_2 & \cdots & \X_{2t}  \\
      \vdots & \vdots & \ddots & \vdots  \\
      0 & 0 & \cdots & \C_t
    \end{array}\right)$$
holds for some basic multiplicative matrices $\C_1,\C_2,\cdots,\C_t$ over $G$.

\item[(2)]
The reason is similar to (1) but noting that $G$ is cosemisimple, which follows that $V$ is a completely irreducible $G$-comodule. In other words, there are simple $G$-comodules $W_1,W_2,\cdots,W_t$ of $V$, such that
$$V=W_1\oplus W_2\oplus\cdots\oplus W_t$$
holds. If we choose linear bases for $W_1,W_2,\cdots,W_t$ respectively, then simple subcoalgebras with basic multiplicative matrices $\C_1,\C_2,\cdots,\C_t$ are obtained as before. The transition matrix $L$ on $V$ from $\{v_1,v_2,\cdots,v_n\}$ to the union of those bases chosen above for $W_1,W_2,\cdots,W_t$ would satisfy the property that
$$L\G L^{-1}=\left(\begin{array}{cccc}
      \C_1 & 0 & \cdots & 0  \\  0 & \C_2 & \cdots & 0  \\
      \vdots & \vdots & \ddots & \vdots  \\  0 & 0 & \cdots & \C_t
    \end{array}\right).$$
\end{itemize}
\end{proof}

Now we turn to mention a binary operations on multiplicative matrices:

\begin{lemma}\label{lem:Kprod}
Suppose $\A=(a_{ij})_{r\times r}$ and $\B=(b_{ij})_{s\times s}$ be multiplicative matrices over a coalgebra $H$. Then
\begin{itemize}
  \item[(1)] The following $rs\times rs$ (block) matrix is multiplicative over the coalgebra $H\otimes H$:
    $$\G:=\left(\begin{array}{ccc}
      a_{11}\otimes\B & \cdots & a_{1r}\otimes\B  \\
      \vdots & \ddots & \vdots  \\
      a_{r1}\otimes\B & \cdots & a_{rr}\otimes\B  \\
    \end{array}\right),
    \;\;\text{where}\;\;
    a_{ij}\otimes\B:=\left(\begin{array}{ccc}
      a_{ij}\otimes b_{11} & \cdots & a_{ij}\otimes b_{1s}  \\
      \vdots & \ddots & \vdots  \\
      a_{ij}\otimes b_{s1} & \cdots & a_{ij}\otimes b_{ss}  \\
    \end{array}\right);$$
  \item[(2)] If $H$ is moreover a bialgebra, then the following $rs\times rs$ matrices are both multiplicative over $H$:
    $$\A\odot\B:=\left(\begin{array}{ccc}
      a_{11}\B & \cdots & a_{1r}\B  \\
      \vdots & \ddots & \vdots  \\
      a_{r1}\B & \cdots & a_{rr}\B  \\
    \end{array}\right)
    \;\;\text{and}\;\;
    \A\odot'\B:=\left(\begin{array}{ccc}
      \A b_{11} & \cdots & \A b_{1s}  \\
      \vdots & \ddots & \vdots  \\
      \A b_{s1} & \cdots & \A b_{ss}  \\
    \end{array}\right).$$
\end{itemize}
\end{lemma}

\begin{remark}
The matrix $\A\odot\B$ is supposed to be called the Kronecker product of $\A$ and $\B$. Clearly, the binary operation $\odot$ could be defined on arbitrary matrices over an algebra in the same ways.
\end{remark}

\begin{proof}
\begin{itemize}
\item[(1)]
Consider the entry $a_{ij}\otimes b_{kl}$ in the block $a_{ij}\otimes\B$. It is direct that
$$\Delta(a_{ij}\otimes b_{kl})=\sum\limits_{r'=1}^r\sum\limits_{s'=1}^s
(a_{ir'}\otimes b_{ks'})\otimes(a_{r'j}\otimes b_{s'l}).$$
Then we compute the entry in $\mathcal{G}\;\widetilde{\otimes}\;\mathcal{G}$ with the same position with $a_{ij}\otimes b_{kl}$ in $\mathcal{G}$. This entry is
\begin{eqnarray*}
&& \sum\limits_{s'=1}^s(a_{i1}\otimes b_{ks'})\otimes(a_{1j}\otimes b_{s'l})
+\sum\limits_{s'=1}^s(a_{i2}\otimes b_{ks'})\otimes(a_{2j}\otimes b_{s'l})  \\
&& +\cdots+\sum\limits_{s'=1}^s(a_{ir}\otimes b_{ks'})\otimes(a_{rj}\otimes b_{s'l})  \\
&=& \sum\limits_{r'=1}^r\sum\limits_{s'=1}^s(a_{ir'}\otimes b_{ks'})\otimes(a_{r'j}\otimes b_{s'l}).
\end{eqnarray*}
In conclusion, $\Delta(\G)=\G\;\widetilde{\otimes}\;\G$. Another requirement $\varepsilon(\G)=I_{rs}$ is evident, since $\varepsilon(a_{ij}\otimes b_{kl})=\delta_{ij}\delta_{kl}$.
\item[(2)]
Note that the multiplication $m:H\otimes H\rightarrow H$ is a coalgebra map. Thus  $\A\odot\B=m(\G)$ is multiplicative.

On the other hand, we consider the bialgebra $H^\mathrm{op}$, whose multiplication is opposite to $H$. It could be seen that $\A$ and $\B$ are still multiplicative over $H^\mathrm{op}$, since $H$ and $H^\mathrm{op}$ share the same coalgebra structures. Therefore, $\A\odot'\B$ is the Kronecker product of $\B$ and $\A$ in $H^\mathrm{op}$ and thus multiplicative.
\end{itemize}
\end{proof}

In the end of this subsection, some evident formulas on such Kronecker products should be noted for later computations:

\begin{lemma}\label{lem:matrixproduct}
Let $H$ be an algebra. Denote the identity matrix of size $n$ by $I_n$.
\begin{itemize}
  \item[(1)] Suppose that $\A_{m_1\times n_1}$ and $\B_{m_2\times n_2}$ are matrices over $H$. Then
      $$(\A\odot\B)^\mathrm{T}=\A^\mathrm{T}\odot\B^\mathrm{T};$$
  \item[(2)] Suppose that $\A_{m_1\times n_1}$, $\B_{m_2\times n_2}$ and $\B'_{n_2\times l_2}$ are matrices over $H$. Then
      $$(\A\odot\B)(I_{n_1}\odot\B')=\A\odot \B\B'.$$
  \item[(3)] If $H$ is furthermore a Hopf algebra with bijective antipode $S$, then for any multiplicative matrix $\G_{n\times n}$ over $H$, we have
      $$S(\G)\G=\G S(\G)=I_n\;\;\;\;\text{and}\;\;\;\;
        S^{-1}(\G)^\mathrm{T}\G^\mathrm{T}=\G^\mathrm{T}S^{-1}(\G)^\mathrm{T}=I_n.$$
\end{itemize}
\end{lemma}

\begin{proof}
Equations in (1) and (2) could be verified directly. The former equation in (3) holds due to the definition of multiplicative matrices. As for the latter one, we compute according to \cite[Lemma 3.4]{LL??1} that
$$S^{-1}(\G)^\mathrm{T}\G^\mathrm{T}=S^{-1}(\G)^\mathrm{T}S^{-1}(S(\G))^\mathrm{T}
  =S^{-1}(S(\G)\G)^\mathrm{T}=S^{-1}(I_n)^\mathrm{T}=I_n,$$
since $S^{-1}$ is an algebra anti-endomorphism on $H$. Of course, $\G^\mathrm{T}S^{-1}(\G)^\mathrm{T}=I_n$ holds similarly.
\end{proof}

\subsection{Non-Trivial Primitive Matrices}

In this subsection, we turn to observe properties of primitive matrices. This notion is a non-pointed analogue of primitive elements (see \cite[Definition 3.2]{LZ19}).

\begin{definition}
Let $(H,\Delta,\varepsilon)$ be a coalgebra over $\k$. Suppose $\C_{r\times r}$ and $\D_{s\times s}$ are basic multiplicative matrices over $H$.
\begin{itemize}
  \item[(1)] An $r\times s$ matrix $\X$ over $H$ is said to be $(\C,\D)$-primitive, if $\Delta(\X)=\C\;\widetilde{\otimes}\;\X+\X\;\widetilde{\otimes}\;\D$;
  \item[(2)] A primitive matrix $\X$ is said to be non-trivial, if any entries of $\X$ does not belong to the coradical $H_0$.
\end{itemize}
\end{definition}

It is clear that entries of primitive matrices must belong to $H_1:=H_0\wedge H_0$. Moreover, there are further properties for non-trivial primitive matrices:

\begin{proposition}\label{prop:nontrivial}
Let $C,D\in\mathcal{S}$, and $\C_{r\times r},\D_{s\times s}$ be their basic multiplicative matrices, respectively. Suppose $\X:=(x_{ij})_{r\times s}$ is a $(\C,\D)$-primitive matrix. Then the followings are equivalent:
\begin{itemize}
  \item[(1)] $\X$ is non-trivial;
  \item[(2)] $x_{ij}\notin H_0$ holds for all $1\leq i\leq r$ and $1\leq j\leq s$;
  \item[(3)] $\{x_{ij}\mid 1\leq j\leq s\}$ are linearly independent in $H_1/H_0$ (the quotient space) for each $1\leq i\leq r$, and $\{x_{ij}\mid 1\leq i\leq r\}$ are linearly independent $H_1/H_0$ for each $1\leq j\leq s$.
\end{itemize}
\end{proposition}

\begin{proof}
Denote that
$$\C=\left(\begin{array}{cccc}
    c_{11} & c_{12} & \cdots & c_{1r}  \\
    c_{21} & c_{22} & \cdots & c_{2r}  \\
    \vdots & \vdots & \ddots & \vdots  \\
    c_{r1} & c_{r2} & \cdots & c_{rr}
  \end{array}\right)
  \;\;\;\;\text{and}\;\;\;\;
  \D=\left(\begin{array}{cccc}
    d_{11} & d_{12} & \cdots & d_{1s}  \\
    d_{21} & d_{22} & \cdots & d_{2s}  \\
    \vdots & \vdots & \ddots & \vdots  \\
    d_{s1} & d_{s2} & \cdots & d_{ss}
  \end{array}\right).$$

(1)$\Rightarrow$(2): Assume (2) does not hold, and that is to say $x_{ij}\in H_0$ for some $i,j$. The condition that $\X$ is $(\C,\D)$-primitive provides the equation
$$\Delta(x_{ij})
  =\sum\limits_{k=1}^r c_{ik}\otimes x_{kj}+\sum\limits_{l=1}^s x_{il}\otimes d_{lj}.$$
Since $\{c_{ik}\mid 1\leq k\leq r\}$ are linearly independent, we could find some linear functions $\{f_{k'}\mid 1\leq k'\leq r\}$ on $H$, such that $\langle f_{k'},c_{ik}\rangle=\delta_{k',k}$ holds for any $1\leq k',k\leq r$. Then we obtain for each $1\leq k\leq r$ that
$$(f_k\otimes\id)\circ\Delta(x_{ij})=x_{ik}+\sum\limits_{l=1}^s\langle f_k,x_{il}\rangle d_{lj},$$
which follows that
$$x_{ik}=(f_k\otimes\id)\circ\Delta(x_{ij})-\sum\limits_{l=1}^s\langle f_k,x_{il}\rangle d_{lj}
\in H_0+D \subseteq H_0,$$
due to our assumption that $x_{ij}\in H_0$.

Obviously there is a similar process on $\{d_{lj}\mid 1\leq l\leq s\}$, and we conclude that the assumption $x_{ij}\in H_0$ would follow that $x_{ik}\in H_0$ and $x_{lj}\in H_0$ hold for all $1\leq k\leq r$ and $1\leq l\leq s$. Consequently it is found that all the entries of $\X$ belong to $H_0$, which contradicts (1).

(2)$\Rightarrow$(3):
For any $1\leq i\leq r$, suppose $\alpha_j\in\k \;(1\leq j\leq s)$ such that $\sum\limits_{j=1}^s\alpha_jx_{ij}\in H_0$. Then from the following computation
\begin{eqnarray*}
\Delta\left(\sum\limits_{j=1}^s\alpha_jx_{ij}\right)
&=& \sum\limits_{j=1}^s\alpha_j\Delta\left(x_{ij}\right)
\;=\; \sum\limits_{j=1}^s \alpha_j
      \left(\sum\limits_{k=1}^r c_{ik}\otimes x_{kj}+\sum\limits_{l=1}^s x_{il}\otimes d_{lj}\right)  \\
&=& \sum\limits_{k=1}^r c_{ik}\otimes\left(\sum\limits_{j=1}^s \alpha_jx_{ij}\right)
+\sum\limits_{j,l=1}^s \alpha_j x_{il}\otimes d_{lj},
\end{eqnarray*}
we know that
$$\sum\limits_{j,l=1}^s \alpha_j x_{il}\otimes d_{lj}
  =\Delta\left(\sum\limits_{j=1}^s\alpha_jx_{ij}\right)
   -\sum\limits_{k=1}^r c_{ik}\otimes\left(\sum\limits_{j=1}^s \alpha_jx_{ij}\right)
  \in H_0\otimes H_0.$$
As a consequence, (2) and the linear independence of $\{d_{lj}\mid 1\leq l,j\leq s\}$ follow that $\alpha_j=0$ for all $1\leq j\leq s$. Thus we conclude that $\{x_{ij}\mid 1\leq j\leq s\}$ are linearly independent in $H_1/H_0$.

The other desired linear independence in $H_1/H_0$ is obtained similarly.

(3)$\Rightarrow$(1): This is direct.
\end{proof}

For the remaining of this paper, each element $x\in H\setminus H_0$ is said to be \textit{non-trivial} for convenience. Moreover, an arbitrary matrix $\X$ over $H$ is also said to be \textit{non-trivial}, if some of its entries does not belong to $H_0$. Of course, they would be called \textit{trivial} otherwise.

\section{Link-Indecomposable Coalgebras and Decompositions}\label{section3}

\subsection{Link Relations and Matric Condition}

The definitions involving \emph{link-indecomposable components} were introduced in \cite{Mon95}.  They were later presented by \cite[Section 4.8]{Rad12} in a slightly different way, which will be listed as follows in this paper. Let $H$ be a coalgebra over $\k$, and denote the set of all its simple subcoalgebras by $\mathcal{S}$. Besides, the wedge product operation on $H$ is denoted by $\wedge$.

\begin{definition}\label{def:link0}
Suppose that $C,D\in\mathcal{S}$.
\begin{itemize}
  \item[(1)] $C$ and $D$ are said to be directly linked in $H$, if $C+D\subsetneq C\wedge D+D\wedge C$;
  \item[(2)] $C$ and $D$ are said to be linked in $H$, if there is an $n\in\N$ and $E_0,E_1,\cdots,E_n\in\mathcal{S}$, such that $C=E_0$, $D=E_n$, and $E_i$ and $E_{i+1}$ are directly linked in $H$ for $0\leq i<n$.
\end{itemize}
\end{definition}

Note that the link relation in $H$ is an equivalence relation on $\mathcal{S}$. It could be remarked that this relation is the same as which in \cite{SM78}. Some relevant concepts and results in the literature are recalled as follows.

\begin{definition}
\begin{itemize}
  \item[(1)] A link-indecomposable subcoalgebra of $H$ is a subcoalgebra $H'\subseteq H$, such that any two simple subcoalgebras of $H'$ are linked in $H'$;
  \item[(2)] A link-indecomposable component of $H$ is a maximal link-indecomposable subcoalgebra of $H$.
\end{itemize}
\end{definition}

It is known that the link-indecomposable components are closely related to the decomposition of coalgebras. This could be seen by following lemmas.

\begin{lemma}(\cite[Lemma 4.8.3]{Rad12})
Suppose $H=H'\oplus H''$ is the direct sum of subcoalgebras $H'$ and $H''$. Let $C,D\in\mathcal{S}$ be simple subcoalgebras of $H$. Then:
\begin{itemize}
  \item[(1)] If $C\subseteq H'$ and $D\subseteq H''$, then $C$ and $D$ are not directly linked in $H$;
  \item[(2)] If $C$ and $D$ are linked in $H$, then $C,D\subseteq H'$ or $C,D\subseteq H''$.
\end{itemize}
\end{lemma}

\begin{lemma}(\cite[Theorem 2.1]{Mon95} and \cite[Theorem 4.8.6]{Rad12})\label{lem:linkdecomp}
\begin{itemize}
  \item[(1)] $H$ is the direct sum of its link-indecomposable components;
  \item[(2)] Suppose that $H=\bigoplus\limits_{i}H_{(i)}$ is the direct sum of non-zero link-indecomposable subcoalgebras of $H$. Then $H_{(i)}$'s are the link-indecomposable components of $H$.
\end{itemize}
\end{lemma}

Now we provide some sufficient conditions for simple subcoalgebras to be linked, with the help of non-trivial matrices over $H$. For the purpose, we introduce a family of so-called \textit{coradical orthonormal idempotents} $\{e_C\}_{C\in\mathcal{S}}$ in $H^\ast$, whose existence is affirmed in \cite[Lemma 2]{Rad78} or \cite[Corollary 3.5.15]{Rad12} for any coalgebra $H$:

\begin{definition}
Let $H$ be a coalgebra. $\{e_C\}_{C\in\mathcal{S}}\subseteq H^\ast$ is called a family of coradical orthonormal idempotents in $H^\ast$, if
$$e_C|_D=\delta_{C,D}\varepsilon|_D,\;\;\;\;e_Ce_D=\delta_{C,D}e_C\;\;\;\;
      (\text{for any}\;C,D\in\mathcal{S}),\;\;\;\;\sum\limits_{C\in\mathcal{S}}e_C=\varepsilon.$$
\end{definition}

It should be remarked that the sum $\sum_{C\in\mathcal{S}}e_C$ is well-defined in $H^\ast$, as long as $\{e_C\}_{C\in\mathcal{S}}$ is orthogonal (hence linearly independent). This is because each $h\in H$ belongs to a finite-dimensional subcoalgebra vanishing through all but finitely many $e_C$'s, as explained in the second paragraph of \cite[Section 1]{Rad78}.

Also, we would use following notations for convenience:
$${^C}h=h\leftharpoonup e_C,\;\;\;h^D=e_D\rightharpoonup h,\;\;\;{^C}h^D=e_D\rightharpoonup h\leftharpoonup e_C\;\;\;(\text{for any}\;h\in H\;\text{and}\;C,D \in \mathcal{S}),$$
where $\leftharpoonup$ and $\rightharpoonup$ are hit actions of $H^\ast$ on $H$. Notations such as $V^{C}:=e_C\rightharpoonup V$ for a subspace $V$ of $H$ are used as well.

It is shown in the next lemma how the coradical orthonormal idempotents are applied to connect non-trivial wedges with non-trivial primitive matrices:

\begin{lemma}\label{lem:linkprim}
Let $C,D\in\mathcal{S}$.
\begin{itemize}
  \item[(1)] Suppose $\{e_E\}_{E\in\mathcal{S}}$ is a family of coradical orthonormal idempotents in $H^\ast$. If $C\wedge D\supsetneq C+D$, then there exists some $x\in C\wedge D$ such that
    $$x={}^Cx{}^D\notin H_0.$$
  \item[(2)] Let $\C,\D$ be basic multiplicative matrices of $C$ and $D$, respectively. Then $C\wedge D\supsetneq C+D$ if and only if there is a non-trivial $(\C,\D)$-primitive matrix over $H$.
\end{itemize}
\end{lemma}

\begin{proof}
\begin{itemize}
\item[(1)] Choose $y\in(C\wedge D)\setminus(C+D)$ and consider the sum
$$y=\varepsilon\rightharpoonup y\leftharpoonup\varepsilon
=\sum\limits_{E,F\in\mathcal{S}}\left(e_F\rightharpoonup y\leftharpoonup e_E\right)
=\sum\limits_{E,F\in\mathcal{S}} {}^Ey{}^F.$$
We claim that:
\begin{itemize}
  \item ${}^Ey{}^F\in D$ holds when $E\neq C$, and
  \item ${}^Ey{}^F\in C$ holds when $F\neq D$.
\end{itemize}
In fact, since $\Delta(y)\in C\otimes H+H\otimes D$, when $E\neq C$ we find that
$${}^Ey{}^F=({}^Ey){}^F\in \left(\langle e_E,C\rangle H+\langle e_E,H\rangle D\right){}^F
  \subseteq D{}^F\subseteq D.$$
The second claim holds similarly.

As a conclusion, we know that the summand ${}^Cy{}^D\notin C+D$, because of our choice of $y$. Now we choose $x:={}^Cy{}^D$. Clearly, $x={}^Cx{}^D$ holds by the fact that $e_C$ and $e_D$ are idempotents. Meanwhile, one could verify that the condition $x\notin C+D$ implies $x\notin H_0$ with a proof by contradiction, according to the hit actions by $\{e_E\}_{E\in\mathcal{S}}$.

\item[(2)]
Suppose that $C\wedge D\supsetneq C+D$ holds, and then there exists some element
$$x={}^Cx{}^D\in (C\wedge D)\setminus H_0$$
according to (1). In fact we could know by direct computations that $x\in {}^CH_1{}^D\setminus H_0$ holds, where $H_1=H_0\wedge H_0$. In order to show that $C$ and $D$ are linked, we might assume $C\neq D$. Therefore, due to \cite[Theorem 3.1]{LZ19}(1), we could obtain a finite number of $(\C,\D)$-primitive matrices, such that $x$ is exactly the sum of some of their entries. Now since $x\notin H_0$ is non-trivial, there must be a non-trivial $(\C,\D)$-primitive matrix $\X$ within, and this could be our desired one.

On the other hand, suppose that $\X$ is a $(\C,\D)$-primitive matrix. It is not hard to know all the entries of a $\X$ must lie in $C\wedge D$. Thus, non-trivial ones would belong to $(C\wedge D)\setminus H_0$, which follows that $C\wedge D\supsetneq C+D$ holds as well.
\end{itemize}
\end{proof}

Lemma \ref{lem:linkprim}(2) could be regarded as a non-pointed generalization of \cite[Lemma 15.2.2]{Rad12}, which provides a condition for simple subcoalgebras to be directly linked.
Furthermore, a sufficient condition for the link relation could also be verified. Before that, we need a lemma on triviality properties of block upper-triangular multiplicative matrices (with basic diagonal) over a coalgebra $H$:

\begin{lemma}\label{lem:matrixtrivial}
Let $\{e_E\}_{E\in\mathcal{S}}$ be a family of coradical orthonormal idempotents in $H^\ast$. Suppose that
\begin{equation}\label{eqn:matrixlink}
  \left(\begin{array}{cccc}
    \C_1 & \X_{12} & \cdots & \X_{1t}  \\
    0 & \C_2 & \cdots & \X_{2t}  \\
    \vdots & \vdots & \ddots & \vdots  \\
    0 & 0 & \cdots & \C_t
  \end{array}\right)
\end{equation}
is a (block) multiplicative matrix over $H$, where $\C_1,\C_2,\cdots,\C_t$ are basic multiplicative matrices for $C_1,C_2,\cdots,C_t\in\mathcal{S}$ respectively. Then
\begin{itemize}
  \item[(1)] $\X_{1t}-{}^{C_1}\X_{1t}{}^{C_t}$ is trivial;
  \item[(2)] ${}^D\X_{1t}$ is trivial for any $D\in\mathcal{S}\setminus\{C_1\}$, and
    $\X_{1t}{}^{D'}$ is trivial for any $D'\in\mathcal{S}\setminus\{C_t\}$.
\end{itemize}
\end{lemma}

\begin{proof}
At first we claim that (2) is a direct consequence of (1). In fact $e_D$ and $e_{C_1}$ are orthogonal in $H^\ast$ when $D\neq C_1$, and then
$${}^D\X_{1t}={}^D\X_{1t}-{}^D({}^{C_1}\X_{1t}{}^{C_t})={}^D(\X_{1t}-{}^{C_1}\X_{1t}{}^{C_t})$$
holds. Thus, ${}^D\X_{1t}$ is trivial because $\X_{1t}-{}^{C_1}\X_{1t}{}^{C_t}$ is so. Similarly, $\X_{1t}{}^{D'}$ is also trivial when $D'\in\mathcal{S}\setminus\{C_t\}$.

Now we try to prove (1) by inductions on $t\geq2$. The case $t=2$ is not hard to verify: Since
$\Delta(\X_{12})=\C_1\;\widetilde{\otimes}\;\X_{12}+\X_{12}\;\widetilde{\otimes}\;\C_2$,
we could obtain ${}^{C_1}\X_{12}=\X_{12}+\langle e_{C_1},\X_{12}\rangle\C_2$ and then
\begin{equation}\label{eqn:X12}
{}^{C_1}\X_{12}{}^{C_2}=\X_{12}+\C_1\langle e_{C_2},\X_{12}\rangle+\langle e_{C_1},\X_{12}\rangle\C_2.
\end{equation}
It follows that $\X_{12}-{}^{C_1}\X_{12}{}^{C_2}$ is trivial as desired.

Assume that (1) holds for $2,3,\cdots,t-1$, and then (2) holds as well. Note that we could actually obtain by the inductive assumption that
$$\X_{ij}-{}^{C_i}\X_{ij}{}^{C_j},\;\;{}^D\X_{ij}\;\;\text{and}\;\;\X_{ij}{}^{D'}\;\;\text{are all trivial},$$
for each $1\leq i<j\leq t$ satisfying $j-i\leq t-2$ and any $D\neq C_i$, $D'\neq C_j$. This is due to the fact
\begin{equation*}
  \left(\begin{array}{cccc}
    \C_i & \X_{i\,i+1} & \cdots & \X_{ij}  \\
    0 & \C_{i+1} & \cdots & \X_{i+1\,j}  \\
    \vdots & \vdots & \ddots & \vdots  \\
    0 & 0 & \cdots & \C_j
  \end{array}\right)
\end{equation*}
is a multiplicative submatrix.

Consider the equation induced by the multiplicative matrix (\ref{eqn:matrixlink}) that
\begin{equation}\label{eqn:DeltaX0}
\Delta(\X_{1t})=\C_1\;\widetilde{\otimes}\;\X_{1t}+\X_{12}\;\widetilde{\otimes}\;\X_{2t}
                  +\cdots+\X_{1t}\;\widetilde{\otimes}\;\C_t,
\end{equation}
which also follows
\begin{eqnarray}\label{eqn:DeltaX1}
\Delta({}^{C_1}\X_{1t}{}^{C_t})
&=& {}^{C_1}\C_1\;\widetilde{\otimes}\;\X_{1t}{}^{C_t}
    +{}^{C_1}\X_{12}\;\widetilde{\otimes}\;\X_{2t}{}^{C_t}
    +\cdots+{}^{C_1}\X_{1t}\;\widetilde{\otimes}\;\C_t{}^{C_t}  \nonumber  \\
&=& \C_1\;\widetilde{\otimes}\;{}^{C_1}\X_{1t}{}^{C_t}
    +{}^{C_1}\X_{12}\;\widetilde{\otimes}\;\X_{2t}{}^{C_t}
    +\cdots+{}^{C_1}\X_{1t}{}^{C_t}\;\widetilde{\otimes}\;\C_t.
\end{eqnarray}
However, we know by computations that
$$\X_{1k}\;\widetilde{\otimes}\;\X_{kt}
=\sum\limits_{E\in\mathcal{S}}\X_{1k}{}^E\;\widetilde{\otimes}\;{}^E\X_{kt}
=\X_{1k}{}^{C_k}\;\widetilde{\otimes}\;{}^{C_k}\X_{kt}
 +\sum\limits_{D\in\mathcal{S}\setminus\{C_k\}}\X_{1k}{}^D\;\widetilde{\otimes}\;{}^D\X_{kt}$$
holds for each $2\leq k\leq t-1$. Hence the inductive assumption implies that all the entries of matrices $$\X_{1k}\;\widetilde{\otimes}\;\X_{kt}
  -\X_{1k}{}^{C_k}\;\widetilde{\otimes}\;{}^{C_k}\X_{kt}\;\;\;\;(2\leq k\leq t-1)$$
belong to $H_0\otimes H_0$. Comparing Equations (\ref{eqn:DeltaX0}) with (\ref{eqn:DeltaX1}), we find that all entries of the matrix
$$\Delta(\X_{1t}-{}^{C_1}\X_{1t}{}^{C_t})
  -\C_1\;\widetilde{\otimes}\;(\X_{1t}-{}^{C_1}\X_{1t}{}^{C_t})
  -(\X_{1t}-{}^{C_1}\X_{1t}{}^{C_t})\;\widetilde{\otimes}\;\C_t$$
would still belong to $H_0\otimes H_0$ as a result. Consequently, a similar process with Equation (\ref{eqn:X12}) provides that
$${}^{C_1}(\X_{1t}-{}^{C_1}\X_{1t}{}^{C_t}){}^{C_t}
  -(\X_{1t}-{}^{C_1}\X_{1t}{}^{C_t})
  -\C_1\langle e_{C_1},\X_{1t}-{}^{C_1}\X_{1t}{}^{C_t}\rangle
  -\langle e_{C_t},\X_{1t}-{}^{C_1}\X_{1t}{}^{C_t})\rangle\C_t$$
is trivial, but in fact
$${}^{C_1}(\X_{1t}-{}^{C_1}\X_{1t}{}^{C_t}){}^{C_t}
={}^{C_1}\X_{1t}{}^{C_t}-{}^{C_1}({}^{C_1}\X_{1t}{}^{C_t}){}^{C_t}=0.$$
We conclude in the end that $\X_{1t}-{}^{C_1}\X_{1t}{}^{C_t}$ must be trivial.
\end{proof}

Finally, we would show that the existence of non-trivial multiplicative matrices of form (\ref{eqn:matrixlink}) are sufficient for the desired link relation.

\begin{proposition}\label{prop:linkprim2}
Suppose $C,D\in\mathcal{S}$.
\begin{itemize}
  \item[(1)] Let $\{e_E\}_{E\in\mathcal{S}}$ be a family of coradical orthonormal idempotents in $H^\ast$. If ${}^CH{}^D\setminus H_0\neq\varnothing$, then $C$ and $D$ are linked;
  \item[(2)] Suppose that
    \begin{equation}
      \left(\begin{array}{cccc}
          \C_1 & \X_{12} & \cdots & \X_{1t}  \\
          0 & \C_2 & \cdots & \X_{2t}  \\
          \vdots & \vdots & \ddots & \vdots  \\
          0 & 0 & \cdots & \C_t
      \end{array}\right)
    \end{equation}
    is a (block) multiplicative matrix over $H$, where $\C_1,\C_2,\cdots,\C_t$ are basic multiplicative matrices for $C_1,C_2,\cdots,C_t\in\mathcal{S}$ respectively. If $\X_{1t}$ is non-trivial, then $C_1$ and $C_t$ are linked.
\end{itemize}
\end{proposition}

\begin{proof}
\begin{itemize}
\item[(1)]
Denote the coradical filtration of $H$ by $\{H_n\}_{n\geq0}$ in this proof. It is evident that
$${}^CH{}^D={}^C (\bigcup\limits_{n\geq 0} H_n){}^D=\bigcup\limits_{n\geq 0} {}^CH_n{}^D,$$
and we would show by induction on $n\geq1$ that for each $C,D\in\mathcal{S}$, $C$ and $D$ are linked if ${}^CH_n{}^D\setminus H_0\neq\varnothing$.

Consider the case when $n=1$. At first we could find that ${}^CH_1{}^D\subseteq C\wedge D$. In fact, this is due to following computations:
\begin{eqnarray*}
\Delta({}^CH_1{}^D) &\subseteq& {}^CH_0\otimes H_1{}^D+{}^CH_1\otimes H_0{}^D  \\
&\subseteq& C\otimes H+H\otimes D.
\end{eqnarray*}
Therefore, $(C\wedge D)\setminus (C+D)\supseteq {}^CH_1{}^D\setminus H_0\neq\varnothing$ holds, which follows that $C$ and $D$ are directly linked.

Now we assume that the above claim holds for $1,2,\cdots,n-1$, and suppose that ${}^CH_n{}^D\setminus H_0\neq\varnothing$. Without the loss of generality, one might moreover assume that ${}^CH_n{}^D\setminus H_1\neq\varnothing$, otherwise ${}^CH_n{}^D={}^CH_1{}^D$ and this case is solved in the previous paragraph. However, we could compute directly to know that
$$\Delta({}^CH_n{}^D)\subseteq
  \sum\limits_{i=0}^n{}^CH_i \otimes H_{n-i}{}^D\subseteq
  \sum\limits_{E\in\mathcal{S}}\sum\limits_{i=0}^n{}^CH_i{}^E\otimes{}^EH_{n-i}{}^D.$$
Discuss the following classified situations:
\begin{itemize}
  \item[a)] There exist some $E\in\mathcal{S}$ and some $1\leq i\leq n-1$ such that
    $${}^CH_i{}^E\setminus H_0\neq\varnothing\;\;\;\;\text{and}\;\;\;\;
      {}^EH_{n-i}{}^D\setminus H_0\neq\varnothing$$
    both hold. Then by our inductive assumption, $C$ and $E$ are linked, and meanwhile $E$ and $D$ are linked.
  \item[b)] For every $E\in\mathcal{S}$ and $1\leq i\leq n-1$, we always have
    $${}^CH_i{}^E\subseteq H_0\;\;\;\;\text{or}\;\;\;\;{}^EH_{n-i}{}^D\subseteq H_0.$$
    This implies that
    \begin{eqnarray*}
    {}^CH_i{}^E\otimes{}^EH_{n-i}{}^D
    &\subseteq& H_0\otimes{}^EH_{n-i}{}^D+{}^CH_i{}^E\otimes H_0  \\
    &\subseteq& H_0\otimes H_n+H_n\otimes H_0
    \end{eqnarray*}
    holds for each $E\in\mathcal{S}$ and $1\leq i\leq n-1$.
    In this situation, we find that
    \begin{eqnarray*}
    \Delta({}^CH_n{}^D)
    &\subseteq& H_0\otimes H_n
     +\sum\limits_{E\in\mathcal{S}}\sum\limits_{i=1}^{n-1}\left({}^CH_i{}^E\otimes{}^EH_{n-i}{}^D\right)
     +H_n\otimes H_0  \\
    &\subseteq& H_0\otimes H_n+H_n\otimes H_0,
    \end{eqnarray*}
    which follows that ${}^CH_n{}^D\subseteq H_0\wedge H_0=H_1$, a contradiction to our additional assumption ${}^CH_n{}^D\setminus H_1\neq\varnothing$.
\end{itemize}
As a conclusion, $C$ and $D$ must be linked.

\item[(2)]
We know by Lemma \ref{lem:matrixtrivial}(1) that $\X_{1t}-{}^{C_1}\X_{1t}{}^{C_t}$ must be trivial. Therefore, ${}^{C_1}\X_{1t}{}^{C_t}$ would also be non-trivial according to our requirement on $\X_{1t}$. On the other hand, evidently all the entries of ${}^{C_1}\X_{1t}{}^{C_t}$ lie in the subspace ${}^{C_1}H{}^{C_t}$, and thus non-trivial ones would belong to ${}^{C_1}H{}^{C_t}\setminus H_0$. It is concluded that $C_1$ and $C_t$ are linked according to (1).
\end{itemize}
\end{proof}

\subsection{Products of Link-Indecomposable Components}

This subsection is devoted to study link-indecomposable components of a (non-pointed) Hopf algebra.
For the purpose and convenience in this paper, we should probably extend the definition of link relations onto arbitrary pairs of subcoalgebras at first. Of course, it coincides with Definition \ref{def:link0} on simple subcoalgebras.

\begin{definition}\label{def:linkgen}
Let $H$ be a coalgebra, and let $H'$, $H''$ be its subcoalgebras. We say that $H'$ and $H''$ are linked, if both of following conditions hold:
\begin{itemize}
  \item For each $C\in\mathcal{S}$ contained in $H'$, there exists an $D\in\mathcal{S}$ contained in $H''$, such that $C$ and $D$ are linked in $H$ (in the sense of Definition \ref{def:link0});
  \item For each $D\in\mathcal{S}$ contained in $H''$, there exists an $C\in\mathcal{S}$ contained in $H'$, such that $C$ and $D$ are linked in $H$.
\end{itemize}
\end{definition}

\begin{remark}\label{rmk:linkcomp}
Suppose that subcoalgebras $H'$ and $H''$ are linked (in the sense of Definition \ref{def:linkgen}). A direct discussion follows that for any $E\in\mathcal{S}$, $H'\cap H_{(E)}\neq 0$ if and only if $H''\cap H_{(E)}\neq 0$.
In particular, $H'$ is linked with some $E\in\mathcal{S}$, if and only if $H'\subseteq H_{(E)}$.
\end{remark}

We turn to consider link relations for a Hopf algebra $H$. We need to mention at first that when the antipode $S$ of is bijective, it is a bijection on $\mathcal{S}$ and $S(H_0)\subseteq H_0$. Now for each $C\in\mathcal{S}$, denote the link-indecomposable component containing $C$ by $H_{(C)}$. The following result is not hard:

\begin{corollary}\label{cor:antipodecomp}
Let $H$ be a Hopf algebra over a field $\k$ with the bijective antipode $S$. Then for any $C\in\mathcal{S}$, $S(H_{(C)})=H_{S(C)}$.
\end{corollary}

\begin{proof}
It is known by \cite[Lemma 15.2.1]{Rad12} that $C_1,C_2\in\mathcal{S}$ are linked, if and only if simple subcoalgebras $S(C_1)$ and $S(C_2)$ are linked. This fact implies that $S(H_{(C)})$ is link-indecomposable and thus contained in $H_{S(C)}$.

On the other hand, the same reason concerning the coalgebra anti-isomorphism $S^{-1}$ follows that $S^{-1}(H_{S(C)})\subseteq H_{(S^{-1}\circ S(C))}=H_{(C)}$, which means that $H_{S(C)}\subseteq S(H_{(C)})$. As a conclusion, $S(H_{(C)})=H_{S(C)}$ holds.
\end{proof}

Now the products of link-indecomposable components of a Hopf algebra could be considered. With the language of Definition \ref{def:linkgen}, we start our process by describing how products of simple subcoalgebras preserve their link relations:

\begin{lemma}\label{lem:prodlink}
Let $H$ be a Hopf algebra over an algebraically closed field $\k$ with the bijective antipode $S$.
\begin{itemize}
  \item[(1)] Suppose $C_1,C_2,D\in\mathcal{S}$, and that $C_1$ and $C_2$ are directly linked. If
    \begin{equation}\label{eqn:prodcond1}
    ((C_1D)_0+(C_2D)_0)(S(D)+S^{-1}(D))\subseteq H_0
    \end{equation}
    holds, then $C_1D$ and $C_2D$ are linked;
  \item[(2)] Suppose $C,D_1,D_2\in\mathcal{S}$, and that $D_1$ and $D_2$ are directly linked. If
    \begin{equation}\label{eqn:prodcond2}
    (S(C)+S^{-1}(C))((CD_1)_0+(CD_2)_0)\subseteq H_0
    \end{equation}
    holds, then $CD_1$ and $CD_2$ are linked.
\end{itemize}
Here $(C_1D)_0$ denotes the coradical of the subcoalgebra $C_1D$, and so on in conditions (\ref{eqn:prodcond1}) and (\ref{eqn:prodcond2}).
\end{lemma}

\begin{proof}
\begin{itemize}
\item[(1)]
Assume that $C_1\wedge C_2\supsetneq C_1+C_2$ without the loss of generality. Suppose $\C_1,\C_2,\D$ are basic multiplicative matrices of $C_1,C_2,D$ with sizes $r_1,r_2,s$, respectively. Then by Lemma \ref{lem:linkprim}(2), there exists a non-trivial $(\C_1,\C_2)$-primitive matrix $\X$. Define a square matrix
$$\G:=
  \left(\begin{array}{cc}
    \C_1 & \X  \\  0 & \C_2
  \end{array}\right)\odot\D=
  \left(\begin{array}{cc}
    \C_1\odot\D & \X\odot\D  \\  0 & \C_2\odot\D
  \end{array}\right).$$
Note that according to Lemma \ref{lem:Kprod}(2), matrices $\C_1\odot\D$, $\C_2\odot\D$ and $\G$ are all multiplicative.

Now we try to observe properties of the $r_1s\times r_2s$ matrix $\X\odot\D$ in details. Of course, each row of $\X\odot\D$ is a vector in $H^{r_1s\times 1}$, which denotes the space of all row vectors with $r_1s$ entries from $H$. Moreover, we might regard these rows as vectors in $(H/H_0)^{r_1s\times 1}$ with entries from the quotient space $H/H_0$. Similar conventions are made for column vectors and spaces $H^{1\times r_2s}$ and $(H/H_0)^{1\times r_2s}$. We aim to show that the following properties (i) and (ii) for the matrix $\X\odot\D$ both hold:

\begin{center}
  \begin{tabular}{cl}
    (i) & The set of all its row vectors is linearly independent over $H/H_0$;  \\
    (ii) & The set of all its column vectors is linearly independent over $H/H_0$.
  \end{tabular}
\end{center}

At first we try to show that $\X\odot\D$ has property (i). Clearly, all the entries of $\X\odot\D$ must belong to $C_1D\wedge C_2D$, and thus trivial ones among them would belong to $(C_1D)_0+(C_2D)_0$.

Assume on the contrary that (i) does not hold for $\X\odot\D$, or equivalently, there is an non-zero $1\times r_1s$ matrix $P$ over $\k$ such that $P(\X\odot\D)$ is trivial as a row vector in $H^{1\times r_2s}$. Moreover, clearly $\X\odot\D$ is actually a matrix over the subcoalgebra $C_1D\wedge C_2D$. Thus all entries of the trivial vector $P(\X\odot\D)$ would belong to $(C_1D)_0+(C_2D)_0$.

However, we could compute by Lemma \ref{lem:matrixproduct}(2) that
$$P(\X\odot\D)(I_{r_2}\odot S(\D))=P(\X\odot\D S(\D))=P(\X\odot I_s),$$
whose entries all lie in $((C_1D)_0+(C_2D)_0)S(D)\subseteq H_0$ due to our condition (\ref{eqn:prodcond1}). This is a contradiction to the fact that $P(\X\odot I_s)$ is a non-trivial row vector, because $\X\odot I_s$ must have property (i) by the definition of our Kronecker product $\odot$ as well as Proposition \ref{prop:nontrivial}(3).

On the other hand, a similar argument would follow that the matrix $(\X\odot\D)^\mathrm{T}$ has property (i) as well. Specifically, for any non-zero $1\times r_2s$ matrix $Q$ over $\k$, we could compute by Lemma \ref{lem:matrixproduct} again to know that:
\begin{eqnarray*}
Q(\X\odot\D)^\mathrm{T}(I_{r_1}\odot S^{-1}(\D))^\mathrm{T}
&=& Q(\X^\mathrm{T}\odot\D^\mathrm{T})(I_{r_1}\odot S^{-1}(\D)^\mathrm{T})  \\
&=& Q(\X^\mathrm{T}\odot\D^\mathrm{T} S^{-1}(\D)^\mathrm{T})  \\
&=& Q(\X^\mathrm{T}\odot I_s),
\end{eqnarray*}
whose entries would all lie in $((C_1D)_0+(C_2D)_0)S^{-1}(D)\subseteq H_0$ by the condition (\ref{eqn:prodcond1}). Of course, this is equivalent to say $\X\odot\D$ has property (ii).

Next we turn to deal with $\G$. It is followed by Proposition \ref{prop:sim}(1) that there exist invertible matrices $L_1$ and $L_2$ over $\k$, such that
\begin{align*}
& L_1(\C_1\odot\D)L_1^{-1}=
  \left(\begin{array}{cccc}
    \E_1 & \Y_{12} & \cdots & \Y_{1t}  \\
    0 & \E_2 & \cdots & \Y_{2t}  \\
    \vdots & \vdots & \ddots & \vdots  \\
    0 & 0 & \cdots & \E_t
  \end{array}\right)\;\;\;\;\text{and}  \\
& L_2(\C_2\odot\D)L_2^{-1}=
  \left(\begin{array}{cccc}
    \F_1 & \Z_{12} & \cdots & \Z_{1u}  \\
    0 & \F_2 & \cdots & \Z_{2u}  \\
    \vdots & \vdots & \ddots & \vdots  \\
    0 & 0 & \cdots & \F_u
  \end{array}\right)
\end{align*}
both hold, where $\E_1,\E_2,\cdots,\E_t$ and $\F_1,\F_2,\cdots,\F_u$ are basic multiplicative matrices over $H$. Meanwhile we denote
$$L_1(\X\odot\D)L_2^{-1}=
  \left(\begin{array}{cccc}
    \X_{11} & \X_{12} & \cdots & \X_{1u}  \\
    \X_{12} & \X_{22} & \cdots & \X_{2u}  \\
    \vdots & \vdots & \ddots & \vdots  \\
    \X_{t1} & \X_{t2} & \cdots & \X_{tu}
  \end{array}\right),$$
where for each $1\leq i\leq t$ and $1\leq j\leq u$, the matrix $\X_{ij}$ has the same number of rows with $\E_i$, and has the same number of columns with $\F_j$. These notations are concluded as follows:
$$\left(\begin{array}{cc}  L_1 & 0  \\  0 & L_2  \end{array}\right)
  \G
  \left(\begin{array}{cc}  L_1^{-1} & 0  \\  0 & L_2^{-1}  \end{array}\right)=
  \left(\begin{array}{cccccc}
    \E_1 & \cdots & \Y_{1t} & \X_{11} & \cdots & \X_{1u}  \\
    \vdots & \ddots & \vdots & \vdots & \ddots & \vdots  \\
    0 & \cdots & \E_t & \X_{t1} & \cdots & \X_{tu}  \\
     &  &  & \F_1 & \cdots & \Z_{1u}  \\
     & 0 &  & \vdots & \ddots & \vdots  \\
     &  &  & 0 & \cdots & \F_u
  \end{array}\right)$$
as multiplicative matrices.

Recall that we has shown that $\X\odot\D$ has properties (i) and (ii). It is then not hard to know that (i) and (ii) both hold for $L_1(\X\odot\D)L_2^{-1}$ as well. Therefore, we could obtain following two facts:
\begin{itemize}
  \item[(I)] For each $1\leq i\leq t$, there is some $1\leq j\leq u$ such that $\X_{ij}$ is non-trivial. Meanwhile,
  \item[(II)] For each $1\leq j\leq u$, there is some $1\leq i\leq t$ such that $\X_{ij}$ is non-trivial.
\end{itemize}

Finally according to Proposition \ref{prop:linkprim2}(2), the non-trivially of $\X_{ij}$ implies that the simple subcoalgebras corresponding to $\E_i$ and $\F_j$ are linked. As a conclusion, $C_1D$ and $C_2D$ are linked in the sense of Definition \ref{def:linkgen}.

\item[(2)]
Consider the opposite Hopf algebra $H^\mathrm{op}$ with multiplication $\cdot^\mathrm{op}$ and antipode $S^{-1}$, where our condition (\ref{eqn:prodcond2}) becomes
$$((D_1\cdot^\mathrm{op}C)_0+(D_2\cdot^\mathrm{op}C)_0)\cdot^\mathrm{op}(S^{-1}(C)+S(C))
  \subseteq H_0.$$
Of course $D_1$ and $D_2$ are also directly linked in $H^\mathrm{op}$, and thus subcoalgebras $D_1\cdot^\mathrm{op}C$ and $D_2\cdot^\mathrm{op}C$ are linked according to (1). This is exactly our desired result.
\end{itemize}
\end{proof}

With Lemma \ref{lem:prodlink}, a sufficient condition for $H_{(1)}$ to be a Hopf subalgebra could be given as follows.

\begin{proposition}\label{prop:H(1)sub}
Let $H$ be a Hopf algebra over an algebraically closed field $\k$ with the bijective antipode $S$. If \begin{equation}\label{eqn:H(1)prodcond}
{(H_{(1)})_0}^3\subseteq H_0
\end{equation}
holds, then $H_{(1)}$ is a Hopf subalgebra.
\end{proposition}

\begin{proof}
Clearly, the unit element $1$ belongs to the subcoalgebra $H_{(1)}$, and we know by Corollary \ref{cor:antipodecomp} that $S(H_{(1)})\subseteq H_{(1)}$ holds as well. It remains to prove that $H_{(1)}$ is closed under the multiplication, which is written as ${H_{(1)}}^2\subseteq H_{(1)}$. However by Remark \ref{rmk:linkcomp}, we only need to show that each simple subcoalgebra $E$ of ${H_{(1)}}^2$ is linked with $\k1$.

In fact, it is known that $(H_{(1)}\otimes H_{(1)})_0=(H_{(1)})_0\otimes(H_{(1)})_0$, since $\k$ is algebraically closed (\cite[Corollary 4.1.8]{Rad12} for example). Consider the multiplication on $H$ as an epimorphism $H_{(1)}\otimes H_{(1)}\rightarrow {H_{(1)}}^2$ of coalgebras, and it is followed by \cite[Corollary 5.3.5]{Mon93} that
$$({H_{(1)}}^2)_0\subseteq{(H_{(1)})_0}^2
  =\sum\limits_{\substack{C\in\mathcal{S}  \\  C\subseteq H_{(1)}}}
   \sum\limits_{\substack{D\in\mathcal{S}  \\  D\subseteq H_{(1)}}} CD.$$
Therefore, each simple subcoalgebra $E$ of ${H_{(1)}}^2$ must be contained in some subcoalgebra $CD$, where $C$ and $D$ are both linked with $\k1$.

Note that condition (\ref{eqn:H(1)prodcond}) and the fact $S(H_{(1)})\subseteq H_{(1)}$ would imply that any triples of simple subcoalgebras of $H_{(1)}$ would satisfy conditions (\ref{eqn:prodcond1}) as well as (\ref{eqn:prodcond2}). As a consequence, for any $C,D\in\mathcal{S}$ linked with $\k1$, we find that $CD$ is linked with $(\k1)^2=\k1$ in final. It is concluded that ${H_{(1)}}^2$ and $\k1$ are linked, and the desired result is obtained.
\end{proof}

In order to study the products for arbitrary link-indecomposable components $H_{(C)}$ and $H_{(D)}$, we might require a stronger condition for $H$. Recall in the literature that a finite-dimensional Hopf algebra $H$ is said to have the dual Chevalley property, if its coradical $H_0$ is a Hopf subalgebra. In this paper, we also use the term \textit{dual Chevalley property} to indicate a Hopf algebra $H$ with its coradical $H_0$ as a Hopf subalgebra, even if $H$ is infinite-dimensional.

Evidently, when the antipode $S$ is bijective, the dual Chevalley property is equivalent to the requirement that ${H_0}^2\subseteq H_0$. On the other hand, the bijectivity of $S$ is in fact a consequence of the dual Chevalley property:

\begin{lemma}(\cite[Corollary 3.6]{Rad77})
Let $H$ be a Hopf algebra. Suppose that $H$ has the dual Chevalley property, which means that its coradical $H_0$ is a Hopf subalgebra. Then the antipode $S$ is bijective.
\end{lemma}

The following direct corollary is due to a similar argument as the end of the proof of Proposition \ref{prop:H(1)sub}. Namely, the dual Chevalley property follows a fact that any triples in $\mathcal{S}$ would satisfy conditions (\ref{eqn:prodcond1}) and (\ref{eqn:prodcond2}), since the antipode is bijective in this case.

\begin{corollary}\label{cor:prodlink}
Let $H$ be a Hopf algebra over an algebraically closed field $\k$ with the dual Chevalley property. Suppose $C_1,C_2,D_1,D_2\in\mathcal{S}$. If $C_1,C_2$ are linked, and $D_1,D_2$ are also linked, then $C_1D_1$ and $C_2D_2$ are linked.
\end{corollary}

Our main result could be a generalized version of \cite[Theorem 3.2(1)]{Mon95}:

\begin{proposition}\label{thm:DCP}
Let $H$ be a Hopf algebra over an algebraically closed field $\k$ with the dual Chevalley property. Then
\begin{itemize}
  \item[(1)] For any $C,D\in\mathcal{S}$,
    $H_{(C)}H_{(D)}\subseteq\sum\limits_{E\in\mathcal{S},\;E\subseteq CD}H_{(E)}$;
  \item[(2)] $H_{(1)}$ is a Hopf subalgebra.
\end{itemize}
\end{proposition}

\begin{proof}
\begin{itemize}
\item[(1)]
The proof is basically similar to which of Proposition \ref{prop:H(1)sub}. By Remark \ref{rmk:linkcomp} as well, it is sufficient to show that each simple subcoalgebra $E'$ of $H_{(C)}H_{(D)}$ is linked with some $E\in\mathcal{S}$ contained in $CD$.

The same argument follows $(H_{(C)}H_{(D)})_0\subseteq (H_{(C)})_0(H_{(D)})_0$ at first, though in fact the dual Chevalley property implies
$$(H_{(C)}H_{(D)})_0=(H_{(C)})_0(H_{(D)})_0
  =\sum\limits_{\substack{C'\in\mathcal{S}  \\  C'\subseteq H_{(C)}}}
   \sum\limits_{\substack{D'\in\mathcal{S}  \\  D'\subseteq H_{(D)}}} C'D'.$$
Therefore, each simple subcoalgebra $E'$ of $H_{(C)}H_{(D)}$ must be contained in some $C'D'$, where $C',C$ are linked, and $D',D$ are linked.

Note that $CD$ is linked with this $C'D'$, according to Corollary \ref{cor:prodlink}. As a consequence, we could know each simple subcoalgebra $E'$ of $H_{(C)}H_{(D)}$ is linked with some $E\in\mathcal{S}$ contained in $CD$, and the desired result is obtained.

\item[(2)]
This is a particular case of Proposition \ref{prop:H(1)sub}, as the condition (\ref{eqn:H(1)prodcond}) would be followed by the dual Chevalley property of $H$.
\end{itemize}
\end{proof}

\begin{remark}
In fact, the assumption that $\k$ is algebraically closed is not necessary. This would be shown as Theorem \ref{thm:DCPnew} in the next section.
\end{remark}

\section{Further Presentations of Link-Indecomposable Components}\label{section4}

In this section, we try to find much more explicit presentations of the link-indecomposable components for a Hopf algebra $H$ with the dual Chevalley property. Our main idea is to identify $H$ with the smash coproduct $H_0\blackltimes H/{H_0}^+H$ as left $H_0$-module coalgebras. However, the link-indecomposable decomposition of $H_0\blackltimes H/{H_0}^+H$ is in fact determined by a semisimple decomposition of $H_0$ as a sense of relative $(H_0,(H_{(1)})_0)$-Hopf bi(co)modules. This process would be stated in the following subsections.

\subsection{Relative Hopf Bi(co)modules over Cosemisimple Hopf Algebras}

We begin with a cosemisimple Hopf algebra $J$ over an arbitrary field $\k$, whose antipode is then bijective according to \cite[Theorem 3.3]{Lar71}. Suppose that $K$ is a Hopf subalgebra of $J$. Clearly $K$ is also cosemisimple, and it is known that $J$ is faithfully flat (or a projective generator) over $K$ by \cite[Theorem 2.1]{Chi14}.

Let ${}^{J}\!\mathcal{M}_K$ (resp. $\mathcal{M}^J_K$) denote the category consisting of \textit{relative $(J,K)$-Hopf modules} which are right $K$-modules as well as left (resp. right) $J$-comodules. One could see \cite[Section 1]{MW94} for details of the definition. Now consider the $\k$-linear abelian category ${}^{J}\!\mathcal{M}^J_K$ consisting of $(J,J)$-bicomodules $M$ equipped with right $K$-module structure which are Hopf modules both in ${}^{J}\!\mathcal{M}_K$ and $\mathcal{M}^J_K$.

\begin{proposition}\label{prop:Hopfmodss}
With notations above, ${}^{J}\!\mathcal{M}^J_K$ is semisimple.
\end{proposition}

\begin{proof}
We aim to show that every short exact sequence
$$0\rightarrow L\xrightarrow{f} M\rightarrow N\rightarrow 0$$
in ${}^{J}\!\mathcal{M}^J_K$ splits. This, regarded as a short exact sequence in ${}^{J}\!\mathcal{M}_K$, splits $K$-linearly since $N$ is $K$-projective by \cite[Corollary 2.9]{MW94}. Consequently, there is a right $K$-module map $g:M\rightarrow L$ such that $gf=\id_L$.

Note that due to the cosemisimplicity (or rather, the coseparability) of $J$, we have a $(J,J)$-bicomodule map $\phi:J\rightarrow\k$ such that $\phi(1)=1$. Moreover, if we denote the $(J,J)$-bicomodule structures of $L$ and $M$ by
\begin{align*}
& \rho_M:M\rightarrow J\otimes M\otimes J,\;m\mapsto\sum m_{(-1)}\otimes m_{(0)}\otimes m_{(1)},  \\
& \rho_L:L\rightarrow J\otimes L\otimes J,\;l\mapsto\sum l_{(-1)}\otimes l_{(0)}\otimes l_{(1)},
\end{align*}
then a retraction of $\rho_L$ (in ${}^{J}\!\mathcal{M}^J_K$) could be obtained as in \cite[Theorem 1]{Doi90}, which is given by
$$r_L:J\otimes L\otimes J\rightarrow L,\;
    a\otimes l\otimes b\mapsto \sum \phi(l_{(-1)}S^{-1}(a))l_{(0)}\phi(l_{(1)}S(b)).$$

Specifically, consider the right $K$-module and $(J,J)$-bicomodule structures of
$J\otimes L\otimes J\in {}^{J}\!\mathcal{M}^J_K$ defined by
$$(a\otimes l\otimes b)\cdot c:=\sum ac_{(1)}\otimes lc_{(2)}\otimes bc_{(3)}
  \;\;\;\text{and}\;\;\;
  a\otimes l\otimes b\mapsto \sum a_{(1)}\otimes a_{(2)}\otimes l\otimes b_{(1)}\otimes b_{(2)}$$
respectively for $a,b\in J$, $c\in K$ and $l\in L$, which clearly make $J\otimes L\otimes J$ an object in ${}^{J}\!\mathcal{M}^J_K$.
One could find that $r_L$ is indeed a morphism in ${}^{J}\!\mathcal{M}^J_K$ satisfying $r_L\rho_L=\id_L$ by direct computations. In fact, for any $a,b\in J$, $c\in K$ and $l\in L$,
\begin{eqnarray*}
r_L((a\otimes l\otimes b)\cdot c)
&=& \sum r_L(ac_{(1)}\otimes lc_{(2)}\otimes bc_{(3)})  \\
&=& \sum\phi\left((lc_{(2)})_{(-1)}S^{-1}(ac_{(1)})\right)(lc_{(2)})_{(0)}
        \phi\left((lc_{(2)})_{(1)}S(bc_{(3)})\right)  \\
&=& \sum\phi\left(l_{(-1)}c_{(2)}S^{-1}(c_{(1)})S^{-1}(a)\right)l_{(0)}c_{(3)}
        \phi\left(l_{(1)}c_{(4)}S(c_{(5)})S(b)\right)  \\
&=& \sum\phi\left(l_{(-1)}S^{-1}(a)\right)l_{(0)}c\phi\left(l_{(1)}S(b)\right)  \\
&=& \sum\phi\left(l_{(-1)}S^{-1}(a)\right)l_{(0)}\phi\left(l_{(1)}S(b)\right)c  \\
&=& r_L(a\otimes l\otimes b)c.
\end{eqnarray*}
We also compute that
\begin{eqnarray*}
&& \rho_L\circ r_L(a\otimes l\otimes b)  \\
&=& \sum\rho_L\left(\phi(l_{(-1)}S^{-1}(a))l_{(0)}\phi(l_{(1)}S(b))\right)  \\
&=& \sum\phi(l_{(-1)}S^{-1}(a))\rho_L(l_{(0)})\phi(l_{(1)}S(b))  \\
&=& \sum\phi(l_{(-2)}S^{-1}(a))l_{(-1)}\otimes l_{(0)}\otimes l_{(1)}\phi(l_{(2)}S(b))  \\
&=& \sum\phi(l_{(-2)}S^{-1}(a_{(3)}))l_{(-1)}S^{-1}(a_{(2)})a_{(1)}\otimes l_{(0)}
        \otimes l_{(1)}S(b_{(2)})b_{(3)}\phi(l_{(2)}S(b_{(1)}))  \\
&=& \sum\phi(l_{(-1)}S^{-1}(a_{(2)}))a_{(1)}\otimes l_{(0)}\otimes b_{(2)}\phi(l_{(1)}S(b_{(1)}))  \\
&=& \sum a_{(1)}\otimes \phi(l_{(-1)}S^{-1}(a_{(2)}))l_{(0)}\otimes \phi(l_{(1)}S(b_{(1)}))b_{(2)}  \\
&=& \sum a_{(1)}\otimes r_L\!\left(a_{(2)}\otimes l\otimes b_{(1)}\right) \otimes b_{(2)}
\end{eqnarray*}
due to the $(J,J)$-bicomodule map $\phi:J\rightarrow\k$. Above computations are essentially the same as the proof of \cite[Lemma (3)(ii) on Page 101]{Doi90}

On the other hand, it is not hard to know that $\rho_M$ and $\id\otimes g\otimes\id$ are both morphisms in ${}^{J}\!\mathcal{M}^J_K$, since $g$ is a right $K$-module map. Consequently, a morphism $r_L(\id\otimes g\otimes\id)\rho_M$ from $M$ to $L$ is found in ${}^{J}\!\mathcal{M}^J_K$.
However, we could know by $gf=\id_L$ that
$$r_L(\id\otimes g\otimes\id)\rho_M\circ f=r_L(\id\otimes g\otimes\id)(\id\otimes f\otimes\id)\rho_L
  =r_L\rho_L=\id_L,$$
and thus the desired splitting is obtained.
\end{proof}

It is evident that $J$ is an object in ${}^{J}\!\mathcal{M}^J_K$. A simple subobject of $J$ in ${}^{J}\!\mathcal{M}^J_K$ is precisely a non-zero minimal right $K$-module subcoalgebra of $J$, which would be called a \textit{simple right $K$-module subcoalgebra} of $J$. In fact, we could give more explicit presentations of these simple right $K$-module subcoalgebras in our situation:

\begin{corollary}\label{cor:modsubcoalg}
Let $J$ and $K$ be as above. Then:
\begin{itemize}
  \item[(1)] $J$ is the direct sum of all simple right $K$-module subcoalgebras of $J$;
  \item[(2)] For every simple subcoalgebra $C$ of $J$, $CK$ is a simple right $K$-module subcoalgebra of $J$. In particular, $K$ is a simple right $K$-module subcoalgebra of $J$;
  \item[(3)] Each simple right $K$-module subcoalgebra $K_i$ is of the form $CK$, where $C$ could be chosen as any simple subcoalgebra contained in $K_i$;
  \item[(4)] Given simple subcoalgebras $C$ and $D$ of $J$, we have that $CK=DK$ if and only if $C\subseteq DK$ if and only if $CK\supseteq D$.
\end{itemize}
\end{corollary}

\begin{proof}
\begin{itemize}
\item[(1)]
This is followed by Proposition (\ref{prop:Hopfmodss}) directly.

\item[(2)]
According to (1), every simple subcoalgebra $C$ of $J$ would be contained in some simple right $K$-module subcoalgebra $K_i$ of $J$. However, $CK$ is clearly a non-zero subobject of $K_i$ in ${}^{J}\!\mathcal{M}^J_K$. Thus $CK=K_i$ holds due to the minimality of $K_i$.

\item[(3)]
The reason is similar to the proof of (2), since $K_i$ must contain some simple subcoalgebra $C$ of $J$.

\item[(4)]
It follows by (2) that $CK$ and $DK$ are both simple right $K$-module subcoalgebras of $J$. The desired claim is then obtained according to (3).
\end{itemize}
\end{proof}

\subsection{A Smash Coproduct and its Link-Indecomposable Components}

We still let $J$ be a cosemisimple Hopf algebra in this subsection, but let $Q$ be a irreducible coalgebra whose coradical $Q_0$ is spanned by a grouplike element $g$. Moreover, suppose $Q$ is a right $J$-comodule coalgebra such that $J$ coacts trivially on $g$, or namely, $g\mapsto g\otimes 1$ under this coaction $Q\rightarrow Q\otimes J$. Then the \textit{smash coproduct} (\cite[Section 4.2]{Doi73}) of $Q$ with $J$ is constructed as
$$H:=J\blackltimes Q,$$
whose coalgebra structure is defined by
$$\Delta(a\blackltimes q)
  :=\sum(a_{(1)}\blackltimes q_{(1)(0)})\otimes(a_{(2)}q_{(1)(1)}\blackltimes q_{(2)})
  \;\;\;\;\text{and}\;\;\;\;
  \varepsilon(a\blackltimes q)=\varepsilon(a)\varepsilon(q)$$
for any $a\in J$ and $q\in Q$. Furthermore, this is evident a left $J$-module coalgebra, via the $J$-module structure
$$J\otimes H\rightarrow H,\;b\otimes (a\blackltimes q)\mapsto ba\blackltimes q.$$
Also, we have following straightforward descriptions of the coradical $H_0$ of $H$:

\begin{lemma}
Suppose $J$, $Q$ are defined as above, and $H:=J\blackltimes Q$. Then:
\begin{itemize}
\item[(1)] The coradical $H_0$ is $J\otimes\k g\cong J$ as coalgebras;
\item[(2)]
If we let
\begin{equation}\label{eqn:projpi}
\pi:H=J\blackltimes Q\rightarrow J\otimes(Q/Q^+)\cong J,\;a\blackltimes q\mapsto a\varepsilon(q)
\end{equation}
denote the natural left $J$-module coalgebra projection, then for any subcoalgebra $H'$ of $H$, $\pi(H')$ is the coradical of $H'$.
\end{itemize}
\end{lemma}

\begin{proof}
\begin{itemize}
\item[(1)]
Since the right $J$-coaction on the unique grouplike $g\in Q$ is trivial, we find that the cosemisimple subcoalgebra $J\otimes\k g$ cogenerates $H$ (via wedge products).
In fact, one could verify that for any $n\geq 1$,
$$\Delta(J\blackltimes Q_n)\subseteq
  (J\blackltimes \k g)\otimes H+H\otimes(J\blackltimes Q_{n-1})$$
holds according to $\Delta(Q_n)\subseteq\k g\otimes Q+Q\otimes Q_{n-1}$, where $\{Q_n\}_{n\geq 0}$ denotes the coradical filtration.

\item[(2)]
Evidently $\pi$ is a left $J$-module coalgebra map.
On the other hand, as an irreducible coalgebra, clearly $Q=\k g\oplus Q^+$ holds. Note that $Q^+$ is a coideal of $Q$. Therefore, $J\blackltimes Q^+$ is also a coideal of $H$, and $H=H_0\oplus(J\blackltimes Q^+)$ holds.
It is not hard to see that the coalgebra map $\pi$ is exactly the projection from $H$ to $H_0$, with respect to this direct sum. Then the desired claim $\pi(H')=H'_0$ could be shown, as a direct application of \cite[Proposition 4.1.7(a)]{Rad12} to $\pi_{H'}$ for example.
\end{itemize}
\end{proof}

Another notion required is the coefficient space for comodules. In general, given a right comodule $V=(V,\rho)$ over a coalgebra $Z$, the \textit{coefficient space} $C_Z(V)$ is the smallest subspace (indeed, subcoalgebra) $Z'$ of $Z$ such that $\rho(V)\subseteq V\otimes Z'$. See \cite[Theorem 3.2.11]{Rad12} for example.

Now with the construction of the smash coproduct $H=J\blackltimes Q$, its indecomposable decomposition of $H$ could be determined under specific assumptions:

\begin{lemma}\label{lem:coproddecomp}
Let $K$ denote the smallest Hopf subalgebra of the cosemisimple Hopf algebra $J$ that includes the coefficient space $C_J(Q)$ of the right $J$-comodule $Q$. With notations above, we have:
\begin{itemize}
\item[(1)] The coalgebra $H$ decomposes as a direct sum of subcoalgebras:
  \begin{equation}\label{eqn:coproddecomp}
  H=\bigoplus_i K_i\blackltimes Q,
  \end{equation}
  where $K_i$ are the simple right $K$-module subcoalgebras of $J$ such that $J=\bigoplus_i K_i$;
\item[(2)] If the coradical $(H_{(1)})_0$ of $H_{(1)}$ is a Hopf subalgebra of $J$, and the coradical of each link-indecomposable component of $H$ is stable under the right multiplication by $(H_{(1)})_0$, then (\ref{eqn:coproddecomp}) gives the link-indecomposable decomposition of the coalgebra $H$.
\end{itemize}
\end{lemma}

\begin{proof}
\begin{itemize}
\item[(1)]
Since $K_i$ is a right $K$-module coalgebra, it then follows from our assumption $K\supseteq C_J(Q)$ that
$$\Delta(K_i\blackltimes Q)\subseteq(K_i\blackltimes Q)\otimes(K_iC_J(Q)\blackltimes Q)
  \subseteq(K_i\blackltimes Q)\otimes (K_i\blackltimes Q)$$
by definition. As a result, $K_i\blackltimes Q$ is a subcoalgebra of $H$. This together with Corollary \ref{cor:modsubcoalg}(1) proves the desired claim.

\item[(2)]
Recall that we have regarded $H_0=J\otimes\k g\cong J$, and hence the coradical of each link-indecomposable component of $H$ is identified as a subcoalgebra of the Hopf algebra $J$.

Suppose that $H'$ is a link-indecomposable component of $H$, and $E$ is a subcoalgebra of the coradical $H'_0$ of $H'$. We claim firstly that the coefficient space $C_J(E\blackltimes Q)$ of the right coideal $E\blackltimes Q$ of $H$, regarded as a right $J$-comodule through $\pi$ (\ref{eqn:projpi}), is included in $H'_0$.

In fact, as a right coideal, $E\blackltimes Q$ is a direct summand of $H$ and hence an injective comodule.
On the other hand, its socle $(E\blackltimes Q)_0$, namely the direct sum of all the simple right $H$-subcomodules, is of course included by $H_0$. Therefore
$$(E\blackltimes Q)_0=\pi((E\blackltimes Q)_0)\subseteq\pi(E\blackltimes Q)=E.$$
As a conclusion, $E\blackltimes Q$ is contained in the injective hull of $E$ as a right coideal by \cite[Proposition 3.5.6(1)]{Rad12}. Hence we could know that $E\blackltimes Q$ is also contained in the component $H'$ according to \cite[Lemma 3.7.1]{Rad12}. This implies that $H'\supseteq C_H(E\blackltimes Q)$, and consequently $H'_0\supseteq C_J(E\blackltimes Q)$ holds as well.

Now consider the link-indecomposable component $H_{(1)}$, and choose the subcoalgebra $E=\k1$. The claim above provides $(H_{(1)})_0\supseteq C_J(Q)$. It follows that if $(H_{(1)})_0$ is a Hopf subalgebra of $J$, then it includes $K=(K\blackltimes Q)_0$. We conclude that $H_{(1)}\supseteq K\blackltimes Q$ as subcoalgebras. Moreover, since $H_{(1)}$ is a indecomposable direct summand of the coalgebra $H$, we then see from (\ref{eqn:coproddecomp}) that
\begin{equation}\label{eqn:H(1)coprod}
H_{(1)}=K\blackltimes Q
\end{equation}
holds, and also $(H_{(1)})_0=K$ as their coradicals.

Finally for any $C\in\mathcal{S}$, assume that the coradical $(H_{(C)})_0$ of the component $H_{(C)}$ is right $K$-stable. We could know according to Corollary \ref{cor:modsubcoalg} that
$$(H_{(C)})_0\supseteq (H_{(C)})_0 K\supseteq CK=K_i$$
for some simple right $K$-module subcoalgebra $K_i$ of $J$.
Then $H_{(C)}=K_i\blackltimes Q$ holds, as is seen from (\ref{eqn:coproddecomp}) by the same argument again.
In a word, each link-indecomposable component of $H$ is of form $K_i\blackltimes Q$, which completes the proof.
\end{itemize}
\end{proof}

\begin{remark}
It is not clear whether $K_i\blackltimes Q$ in (\ref{eqn:coproddecomp}) is indecomposable without the assumptions of (2), so that the decomposition above may not be indecomposable.
\end{remark}

\subsection{Presentations of Components with the Dual Chevalley Property}

In this subsection, $H$ is assumed to be a Hopf algebra with the dual Chevalley property over an arbitrary field $\k$. For convenience, let $J=H_0$ be its coradical, and let
$$Q:=H/J^+H,$$
which is clearly a quotient right $H$-module coalgebra of $H$. In addition, $Q$ is a irreducible coalgebra with the one-dimensional coradical $Q_0=\k\overline{1}$ spanned by the natural image $\overline{1}$ of $1\in H$, because $Q$ is cogenerated by
$$\overline{H_0}=H_0/(H_0^+H\cap H_0)=H_0/H_0^+=\k\overline{1}.$$
Furthermore, it could be known by \cite[Theorem 3.1]{Mas03} that $H$ is isomorphic to the smash coproduct $J\blackltimes Q$ constructed. We recall this result with a bit more details, all of which could be found in \cite[Section 4]{Mas03}.

\begin{lemma}(\cite[Theorem 3.1]{Mas03})
Let $H$ be a Hopf algebra with the dual Chevalley property. Denote $J:=H_0$ and $Q:=H/H_0^+H$ respectively as above.
Then:
\begin{itemize}
\item[(1)]
There exists a left $J$-module coalgebra retraction $\gamma:H\rightarrow J$ of the inclusion $J\hookrightarrow H$;
\item[(2)]
$Q$ turns into a right $J$-comodule coalgebra through the structure $\rho:Q\rightarrow Q\otimes J$ induced from
$$H\rightarrow Q\otimes J,\;h\mapsto \sum\overline{h_{(2)}}\otimes S(\gamma(h_{(1)}))\gamma(h_{(3)}),$$
where $h\mapsto \overline{h}$ represents the projection $H\twoheadrightarrow Q$;
\item[(3)]
The resulting left $J$-module coalgebra of smash coproduct of $Q$ with $J$ is isomorphic to $H$ through
\begin{equation}\label{eqn:isocoprod}
H\xrightarrow{\simeq} J\blackltimes Q,\;h\mapsto \sum\gamma(h_{(1)})\otimes \overline{h_{(2)}}.
\end{equation}
\end{itemize}
\end{lemma}

We remark that by taking the associated grading $\mathsf{gr}$ with respect to the coradical filtration, the isomorphism (\ref{eqn:isocoprod}) turns into the canonical isomorphism
$$\textsf{gr}\;H\xrightarrow{\simeq} J\dotltimes R$$
of the graded Hopf algebra $\mathsf{gr}\;H$ onto the associated bosonization, where $R=\mathsf{gr}\;Q$, a Nichols algebra in $\mathcal{YD}^J_J$. Since $R$ is generated by $(\mathsf{gr}\;Q)(1)=P(Q)$, the space of the primitives in $Q$, it follows that the smallest Hopf subalgebras of $J$ that includes the following three coefficient spaces
$$C_J(Q),\;\;C_J(R),\;\;C_J(P(H))$$
coincide.
We let $K$ denote the coinciding Hopf subalgebra.

The following proposition could be proved with the usage of Proposition \ref{thm:DCP} over the algebraic closure $\overline{\k}$ of the base field $\k$.

\begin{proposition}\label{prop:Hdecomp}
Let $H$ be a Hopf algebra with the dual Chevalley property over $\k$. Denote $J:=H_0$ and $Q:=H/H_0^+H$ respectively. Then (\ref{eqn:coproddecomp}) gives the link-indecomposable decomposition of $H$.
\end{proposition}

\begin{proof}
We aim to verify the assumption of Lemma \ref{lem:coproddecomp}(2) for $H$. Firstly, consider the Hopf algebra $H\otimes\overline{\k}$ over the algebraically closed field $\overline{\k}$. Note that $H\otimes\overline{\k}$ also has the dual Chevalley property, since it is cogenerated by its cosemisimple Hopf subalgebra $J\otimes\overline{\k}$ (\cite[Lemma 1.3]{Lar71}).

Now it follows from Proposition \ref{thm:DCP}(2) that the coradical of $(H\otimes\overline{\k})_{(1)}$ is a Hopf $\overline{\k}$-subalgebra of $(H\otimes\overline{\k})_0=J\otimes\overline{\k}$.
Besides, since
$$(H\otimes\overline{\k})/(J\otimes\overline{\k})^+(H\otimes\overline{\k})
  =(H/J^+H)\otimes\overline{\k}=Q\otimes\overline{\k},$$
the smallest Hopf subalgebra of $H\otimes\overline{\k}$ including $C_{J\otimes\overline{\k}}(Q\otimes\overline{\k})$ is exactly $K\otimes\overline{\k}$.
Thus we see from Equation (\ref{eqn:H(1)coprod}) on $H\otimes\overline{\k}$ in the proof of Lemma \ref{lem:coproddecomp}(2) that
\begin{equation}\label{eqn:Hbaseext(1)}
(H\otimes\overline{\k})_{(1)}
=(K\otimes\overline{\k})\;\blackltimes\nolimits_{\overline{\k}}\;(Q\otimes\overline{\k})
=(K\blackltimes Q)\otimes\overline{\k},
\end{equation}
which is indecomposable as a subcoalgebra of $H\otimes\overline{\k}$. Therefore, $K\blackltimes Q$ in $H$ must be indecomposable, and this is exactly $H_{(1)}$ because we know that $(H\otimes\overline{\k})_{(1)}\subseteq H_{(1)}\otimes\overline{\k}$ holds. As a result, we have $(H_{(1)})_0=K$, a Hopf subalgebra of $J$.

On the other hand, applying the projection (\ref{eqn:projpi}) on the subcoalgebra $(H\otimes\overline{\k})_{(1)}$, we know by Equation (\ref{eqn:Hbaseext(1)}) that its coradical is
$$((H\otimes\overline{\k})_{(1)})_0=K\otimes\overline{\k}.$$
It then follows from Proposition \ref{thm:DCP}(1) that each link-indecomposable component of $H\otimes\overline{\k}$ is right $K\otimes\overline{\k}$-stable. Consequently, a component of $H$ turns, after the base extension $\otimes\overline{\k}$, into a direct sum of some components in $H\otimes\overline{\k}$, which are right $K\otimes\overline{\k}$-stable. Therefore, the original component of $H$ and its coradical as well, are right $K$-stable.
\end{proof}

Now we could improve the presentation for the link-indecomposable components of $H$ as follows, where $\mathcal{S}$ still denotes the set of all the simple subcoalgebras:

\begin{theorem}\label{thm:DCPnew}
Let $H$ be a Hopf algebra over an arbitrary field $\k$ with the dual Chevalley property. Then
\begin{itemize}
  \item[(1)] For any $C\in\mathcal{S}$, $H_{(C)}=CH_{(1)}=H_{(1)}C$;
  \item[(2)] For any $C,D\in\mathcal{S}$,
    $H_{(C)}H_{(D)}\subseteq\sum\limits_{E\in\mathcal{S},\;E\subseteq CD}H_{(E)}$;
  \item[(3)] $H_{(1)}$ is a Hopf subalgebra.
\end{itemize}
\end{theorem}

\begin{proof}
\begin{itemize}
\item[(1)]
Combining Proposition \ref{prop:Hdecomp} and Corollary \ref{cor:modsubcoalg} with notations in this subsection, we know that $H_{(C)}=CK\blackltimes Q$ holds for any $C\in\mathcal{S}$. Moreover, it follows from the identification (\ref{eqn:isocoprod}) $H=J\blackltimes Q$ as left $J$-module coalgebras that
$$H_{(C)}=CK\blackltimes Q=C(K\blackltimes Q)=CH_{(1)}.$$
The other equation $H_{(C)}=H_{(1)}C$ is obtained by the opposite-sided results above.

\item[(2)]
This is followed by (1), namely,
$$H_{(C)}H_{(D)}=CH_{(1)}DH_{(1)}=CDH_{(1)}
  =\left(\sum\limits_{E\in\mathcal{S},\;E\subseteq CD}E\right)H_{(1)}
  \subseteq\sum\limits_{E\in\mathcal{S},\;E\subseteq CD}H_{(E)}$$
for any $C,D\in\mathcal{S}$.

\item[(3)]
This is directly followed by (2) and Corollary \ref{cor:antipodecomp}.
\end{itemize}
\end{proof}

\begin{remark}
Lemma \ref{lem:prodlink} as well as Theorem \ref{thm:DCPnew} might fail for a Hopf algebra $H$ without the dual Chevalley property. A counter-example is presented as Example \ref{ex:withoutDCP} in the next section.
\end{remark}

In addition, we introduce an equivalence relation on $\mathcal{S}$, defining that $C$ and $D$ are related if $CK=DK$ (or equivalently $KC=KD$), where $K=(H_{(1)})_0$ as before. Let $\mathcal{S}_0\subseteq\mathcal{S}$ be a full set of chosen non-related representatives with respect to this equivalence relation. Then a presentation of the indecomposable decomposition of $H$ is stated as follows:

\begin{corollary}\label{cor:Hdecomp}
Let $H$ be a Hopf algebra over an arbitrary field $\k$ with the dual Chevalley property. Then the link-indecomposable decomposition of $H$ could be presented as
$$H=\bigoplus_{C\in\mathcal{S}_0} CH_{(1)}.$$
\end{corollary}

\begin{proof}
It suffices to show that for any $C,D\in\mathcal{S}$, $CK=DK$ if and only if $C$ and $D$ are linked. This
could be known from the fact $CK\blackltimes Q=H_{(C)}$ appearing in the proof of Theorem \ref{thm:DCPnew}(1) which holds for each $C\in\mathcal{S}$.
\end{proof}

This corollary partially generalizes Montgomery's result \cite[Theorem 3.2(4)]{Mon95} (and its proof) for a pointed Hopf algebra $H$:
$$H=\bigoplus_{g\in G(H)/G(H_{(1)})} gH_{(1)},$$
where $G(H)$ and $G(H_{(1)})$ denote the set of all the grouplikes of $H$ and $H_{(1)}$, respectively. Indeed when $H$ is pointed, simple subcoalgebras spanned by $g,g'\in G(H)$ are related (in the sense preceding Corollary \ref{cor:Hdecomp}) if and only if $gG(H_{(1)})=g'G(H_{(1)})$. We also remark that $G(H_{(1)})$ is a normal subgroup of $G(H)$ by \cite[Theorem 3.2(3)]{Mon95}.

One more corollary to mention is that the construction of $H_{(1)}$ is compatible with base extension:

\begin{corollary}
Let $H$ be a Hopf algebra with the dual Chevalley property over $\k$. Then for any field extension $\mathbb{F}/\k$,
$$(H\otimes\mathbb{F})_{(1)}=H_{(1)}\otimes\mathbb{F}.$$
\end{corollary}

\begin{proof}
Let $\mathbb{F}$ replace $\overline{\k}$ in the proof of Proposition \ref{prop:Hdecomp}. Then by Theorem \ref{thm:DCPnew}, the coradical of $(H\otimes\mathbb{F})_{(1)}$ is also a Hopf subalgebra. Consequently, equations as (\ref{eqn:Hbaseext(1)})
$$(H\otimes\mathbb{F})_{(1)}=(K\blackltimes Q)\otimes\mathbb{F}$$
and $H_{(1)}=K\blackltimes Q$ hold in this situation as well.
\end{proof}

\subsection{Properties for $H$ over the Hopf Subalgebra $H_{(1)}$}\label{subsection:HoverH(1)}

At the final of this section, the faithfully flatness and freeness of $H$ as an $H_{(1)}$-module are discussed when $H$ has the dual Chevalley property.
In fact, the faithful flatness could be replaced by a stronger property that $H$ is a projective generator as a $H_{(1)}$-module:

\begin{corollary}\label{cor:ffness}
Let $H$ be a Hopf algebra with the dual Chevalley property. Then:
\begin{itemize}
  \item[(1)] $H$ is a projective generator of left as well as right $H_{(1)}$-modules;
  \item[(2)] For any $C\in\mathcal{S}$, $H_{(C)}$ is a finitely generated projective generator of left as well as right $H_{(1)}$-modules.
\end{itemize}
\end{corollary}

\begin{proof}
\begin{itemize}
\item[(1)]
By Lemma \ref{lem:linkdecomp}, there is a direct sum $H=H_{(1)}\oplus M$, where $M$ denotes the direct sum of all link-indecomposable components of $H$ excluding $H_{(1)}$. It follows by Theorem \ref{thm:DCPnew}(2) that $H_{(1)}M$ and $MH_{(1)}$ are both contained in $M$.
The faithful flatness of $H$ over $H_{(1)}$ is then obtained according to a combination of \cite[Propositions 1.4 and 1.6]{Chi14}.

Also, recall that Hopf algebras $H$ and $H_{(1)}$ must have bijective antipodes, since they both have the dual Chevalley property. As a result, we could find by \cite[Corollary 2.9]{MW94} that $H$ is even a projective generator as a left as well as right $H_{(1)}$-module.

\item[(2)]
Clearly, it follows by Theorem \ref{thm:DCPnew}(1) that each $H_{(C)}$ is finitely generated over $H_{(1)}$ (on both sides).
On the other hand, since each $H_{(C)}$ is a non-zero object in ${}_{H_{(1)}}\mathcal{M}^H$ as well as ${}^H\!\mathcal{M}_{H_{(1)}}$, the desired result is obtained from (1) according to \cite[Corollary 2.9]{MW94} as well.
\end{itemize}
\end{proof}

However, we could show that $H$ is not always free over $H_{(1)}$ by an example in the following.

We begin with an example of commutative cosemisimple Hopf algebras $J$ which is not free over some of its Hopf subalgebra $K$ (see \cite[Section 5]{Tak79} e.g.). One could choose also a right $K$-comodule $V$, such that the coefficient space $C_K(V)$ and its image $S(C_K(V))$ under the antipode $S$ generate $K$. Let $R$ be the symmetric algebra on $V$. It is routine to verify that via diagonal $K$-coactions, $R$ forms a right $K$-comodule Hopf algebra containing the elements of $V$ as primitives. According to \cite[Section 4]{AD72}, the construction of smash coproduct of $R$ with $K$ gives rise to a commutative Hopf algebra
$$K\blackltimes R,$$
which is the tensor product $K\otimes R$ as an algebra.

Besides, since $R$ is naturally a right $J$-comodule Hopf algebra, we have the analogous Hopf algebra
$$H:=J\blackltimes R,$$
which includes $K\blackltimes R$ as a Hopf subalgebra. Evidently $H$ has coradical $J\otimes 1$ which is also a Hopf subalgebra.
Moreover, note that $R$ is irreducible as a coalgebra with grouplike element $1$ on which $J$ coacts trivially, and $K$ is in fact the smallest Hopf subalgebra of $J$ that includes $C_J(R)$ by our choice of $V$. Thus we know that $H_{(1)}=K\blackltimes R$, as is seen from Proposition \ref{prop:Hdecomp} as well as Lemma \ref{lem:coproddecomp}. We conclude that this $H$ is not free over $H_{(1)}$. Otherwise, if it were free, the tensor product $\otimes_R R/R^+$ would imply that $J$ were free over $K$, a contradiction to our chosen example.

Finally, recall in \cite[Theorem 3.2]{Mon93} that $H_{(1)}$ must be a normal Hopf subalgebra for any pointed Hopf algebra $H$. When $H$ is non-pointed, we pose the following question on the normality of $H_{(1)}$, although there are several positive examples such as $T_\infty(2,1,-1)^\circ$ presented in Subsection \ref{subsection:T},

\begin{question}
Let $H$ be a Hopf algebra with the dual Chevalley property. Is the Hopf subalgebra $H_{(1)}$ normal in $H$?
\end{question}

It is not clear to us whether the answer to this question is positive or not, even for $\mathsf{gr}\;H$, for which $(\mathsf{gr}\;H)_{(1)}$ is normal in $\mathsf{gr}\;H$ if and only if $K$ is normal in $J$. This last condition should be related with possible structures on $P(R)$ as a Yetter-Drinfeld module, or namely, as an object in $\mathcal{YD}^J_J$.

\section{Examples}\label{section5}

For the remaining of this paper, $\k$ is always assumed to be an algebraically closed field of characteristic $0$. Before specific examples, we provide an evident lemma which helps us determine link-indecomposable components:

\begin{lemma}\label{lem:linkcomp}
Let $H$ be a coalgebra, and $\C_1,\C_2,\cdots,\C_t$ be basic multiplicative matrices of $C_1,C_2,\cdots,C_t\in\mathcal{S}$, respectively. Suppose that there is a multiplicative matrix of form
\begin{equation}\label{eqn:linkcomp}
\G:=\left(\begin{array}{cccc}
    \C_1 & \X_{12} & \cdots & \X_{1t}  \\
    0 & \C_2 & \cdots & \X_{2t}  \\
    \vdots & \vdots & \ddots & \vdots  \\
    0 & 0 & \cdots & \C_t  \\
\end{array}\right).
\end{equation}
If $C_1,C_2,\cdots,C_t$ are linked, then all the entries of $\G$ belong to this link-indecomposable component $H_{(C_1)}$.
\end{lemma}

\begin{proof}
Since $\G$ is multiplicative, all its entries would span a subcoalgebra $H'$. Also, $C_1,C_2,\cdots,C_t$ are exactly all the simple subcoalgebras of $H'$. Thus if $C_1,C_2,\cdots,C_t$ are linked, then $H'$ is link-indecomposable and thus contained in the link-indecomposable component.
\end{proof}

\subsection{Without the Dual Chevalley Property}\label{subsection:D}

As mentioned in the end of Section \ref{section3}, the dual Chevalley property might be necessary for Lemma \ref{lem:prodlink} or Corollary \ref{cor:prodlink} in a way. We would show that the following Hopf algebra, denoted by $D(2,2,\sqrt{-1})$, does not satisfy the property in Lemma \ref{lem:prodlink}. The structure is in fact a particular example of a certain classification $D(m,d,\xi)$ introduced in \cite[Section 4.1]{WLD16}, where $m$ and $d$ are both chosen to be $2$.

\begin{example}\label{ex:withoutDCP}
Let $\sqrt{-1}$ be a fixed square root of $-1$. As an algebra, $D(2,2,\sqrt{-1})$ is generated by $x^{\pm1},g^{\pm1},y,u_0,u_1$ with relations:
\begin{align*}
& xx^{-1}=x^{-1}x=1,\;\;gg^{-1}=g^{-1}g=1,  \\
& xy=yx,\;\;gx=xg,\;\;yg=-gy,\;\;y^2=1-x^4=1-g^2,  \\
& u_ix=x^{-1}u_i,\;\;u_ig=(-1)^ig^{-1}u_i,\;\;yu_i=(1+(-1)^ix^2)u_{1-i}=\sqrt{-1}x^2u_iy
\end{align*}
for $i=0,1$, and
\begin{align*}
& u_0^2=\frac{1}{2}x(1+x^2)g^{-1},\;\;u_0u_1=\frac{\sqrt{-1}}{2}xg^{-1}y,\;\;
  u_1u_0=-\frac{1}{2}xg^{-1}y,\;\;u_1^2=-\frac{\sqrt{-1}}{2}x(1-x^2)g^{-1}.
\end{align*}
The coalgebra structure and antipode are given by:
\begin{align*}
& \Delta(x)=x\otimes x,\;\;\Delta(g)=g\otimes g,\;\;\Delta(y)=1\otimes y+y\otimes g,  \\
& \Delta(u_0)=u_0\otimes u_0-u_1\otimes x^{-2}gu_1,\;\;
  \Delta(u_1)=u_0\otimes u_1+u_1\otimes x^{-2}gu_0,  \\
& \varepsilon(x)=\varepsilon(g)=\varepsilon(u_0)=1,\;\;\varepsilon(y)=\varepsilon(u_1)=0,  \\
& S(x)=x^{-1},\;\;S(g)=g^{-1},\;\;S(y)=g^{-1}y,\;\;
  S(u_0)=x^{-3}gu_0,\;\;S(u_1)=-\sqrt{-1}x^{-1}u_1.
\end{align*}
\end{example}

With the application of the Diamond Lemma \cite{Ber78}, we could know that $D(2,2,\sqrt{-1})$ has a linear basis
\begin{equation}\label{eqn:Dbasis}
\{x^ig^jy^l \mid 0\leq i\leq 3,\;j\in\mathbb{Z},\;0\leq l\leq 1\}
\cup\{x^ig^ju_l\mid i\in\mathbb{Z},\;0\leq j,l\leq 1\}.
\end{equation}
An equivalent but more general version is \cite[Lemma 3.3]{Wu16}, but we write the basis in this form (\ref{eqn:Dbasis}) for our purposes. Furthermore, all the simple subcoalgebras and their basic multiplicative matrices are also needed:

\begin{proposition}
The set of all the simple subcoalgebras of $D(2,2,\sqrt{-1})$ is
$$\mathcal{S}=\{\k x^ig^j\mid 0\leq i\leq 3,\;j\in\mathbb{Z}\}
              \cup\{x^iC\mid i\in\mathbb{Z}\},$$
where $C:=\k\{x^{-2j}g^ju_l\mid 0\leq j,l\leq 1\}$ with a basic multiplicative matrix
$$\C:=\left(\begin{array}{cc}
    u_0 & u_1  \\  -x^{-2}gu_1 & x^{-2}gu_0
  \end{array}\right),$$
and $x^iC\neq x^{i'}C$ as long as $i\neq i'$.
\end{proposition}

\begin{proof}
Verified by the structure of $D(2,2,\sqrt{-1})$ and direct computations. One could see \cite[Proposition 3.2]{Wu16} for more general cases.
\end{proof}

Now we know that $D(2,2,\sqrt{-1})$ does not have the dual Chevalley property, since for $u_0,u_1\in C$, their products $u_0u_1$ and $u_1u_0$ do not belong to the coradical.

\begin{proposition}\label{prop:Ddecomp}
The link-indecomposable decomposition of $H:=D(2,2,\sqrt{-1})$ is
$$H=\left(\bigoplus\limits_{0\leq i\leq 3}H_{(x^i)}\right)
    \oplus\left(\bigoplus\limits_{i\in\mathbb{Z}}H_{(x^iC)}\right),$$
where $H_{(x^i)}=\k\{x^ig^jy^l\mid 0\leq j,l\leq 1\}=x^iH_{(1)}$ and $H_{(x^iC)}=x^iC$.
\end{proposition}

\begin{proof}
On the one hand, note that $\Delta(x^ig^jy)=x^ig^j\otimes x^ig^jy+x^ig^jy\otimes x^ig^{j+1}$ always holds. Thus, for each fixed $0\leq i\leq 3$, the simple subcoalgebras (or grouplike elements)
$$\cdots,\;x^ig^{-2},\;x^ig^{-1},\;x^i,\;x^ig,\;x^ig^2,\;x^ig^3,\;\cdots,$$
or equivalently
$$\cdots,\;x^{i-4},\;x^{i-4}g,\;x^i,\;x^ig,\;x^{i+4},\;x^{i+4}g,\;\cdots,$$
are linked, and $x^ig^jy$ belongs to this link-indecomposable component $H_{(x^i)}$. We conclude that
\begin{equation}\label{eqn:Ddecomp1}
\k\{x^ig^jy^l\mid j\in\mathbb{Z},\;0\leq l\leq 1\}
  \subseteq H_{(x^i)}\;\;\;\;(0\leq i\leq 3).
\end{equation}

On the other hand, the remaining non-pointed simple subcoalgebras clearly satisfy
\begin{equation}\label{eqn:Ddecomp2}
\k\{x^{i-2j}g^ju_l\mid 0\leq j,l\leq 1\}=x^iC
  \subseteq H_{(x^iC)}\;\;\;\;(i\in\mathbb{Z}).
\end{equation}

However, the direct sum of the left-hand sides of (\ref{eqn:Ddecomp1}) and (\ref{eqn:Ddecomp2}) become exactly $D(2,2,\sqrt{-1})$, according to the form of the basis (\ref{eqn:Dbasis}). The desired link-indecomposable decomposition is then obtained as the direct sum of the right-hand sides.
\end{proof}

\begin{remark}
Note that as a pointed subcoalgebra, $H_{(1)}$ would satisfy condition (\ref{eqn:H(1)prodcond}). Thus it is a Hopf subalgebra, even though $H=D(2,2,\sqrt{-1})$ does not have the dual Chevalley property.
\end{remark}

Finally we could verify that $D(2,2,\sqrt{-1})$ does not have the property in Lemma \ref{lem:prodlink}(1). Consider simple subcoalgebras $\k1$, $\k g$ and $C$, and note that
$$C=\k\{x^{-2j}g^ju_l\mid 0\leq j,l\leq 1\}=\k\{x^{2j}g^{-j}u_l\mid 0\leq j,l\leq 1\}$$
holds since $x^{-2}g=x^2g^{-1}$.
Clearly, $\k1$ and $\k g$ are linked, but we could compute that
\begin{eqnarray*}
gC
&=& g\cdot\k\{x^{2j}g^{-j}u_l\mid 0\leq j,l\leq 1\}
~=~ \k\{x^{2j}g^{1-j}u_l\mid 0\leq j,l\leq 1\}  \\
&=& \k\{x^{2j}g^{1-j}u_l\mid 0\leq 1-j,l\leq 1\}
~=~ \k\{x^{2(1-j)}g^{j}u_l\mid 0\leq j,l\leq 1\}  \\
&=& \k\{x^{2(1-j)}g^{2j}g^{-j}u_l\mid 0\leq j,l\leq 1\}
~=~ \k\{x^{2(1-j)}x^{4j}g^{-j}u_l\mid 0\leq j,l\leq 1\}   \\
&=& \k\{x^{2+2j}g^{-j}u_l\mid 0\leq j,l\leq 1\}
~=~ x^2C,
\end{eqnarray*}
which is not linked with $C$. That is to say, $(\k1)C$ and $(\k g)C$ are \textit{not} linked, and hence the property in Lemma \ref{lem:prodlink}(1) does not hold.

Moreover, one could find by direct computations that
$$H_{(1)}H_{(C)}=H_{(C)}\oplus H_{(x^2C)}\nsubseteq H_{(C)}$$
for example. Thus $D(2,2,\sqrt{-1})$ dissatisfies the properties in Items (1) and (2) of Theorem \ref{thm:DCPnew}.

\subsection{Non-Degenerate Hopf Pairings}\label{subsection:T}

When $H$ is infinite-dimensional, sometimes $H_{(1)}$ could be an idea for constructing non-degenerate Hopf pairings. The notion of pairings of bialgebras or Hopf algebras are due to \cite{Maj90}. This is also regarded as a sense of a quantum group in \cite{Tak92}.

\begin{definition}
Let $H$ and $H^\bullet$ be Hopf algebras.
A linear map $\langle,\rangle:H^\bullet\otimes H\rightarrow\k$ is called a Hopf pairing (on $H$), if
$$\begin{array}{ll}
\mathrm{(i)}\;\;\;\;\langle ff',h\rangle=\sum\langle f,h_{(1)}\rangle\langle f',h_{(2)}\rangle,
& \mathrm{(ii)}\;\;\;\;\langle f,hh'\rangle=\sum\langle f_{(1)},h\rangle\langle f_{(2)},h'\rangle,  \\
\mathrm{(iii)}\;\;\langle 1,h\rangle=\varepsilon(h),
& \mathrm{(iv)}\;\;\;\langle f,1\rangle=\varepsilon(f),  \\
\mathrm{(v)}\;\;\;\langle f,S(h)\rangle=\langle S(f),h\rangle
\end{array}$$
hold for all $f,f'\in H^\bullet$ and $h,h'\in H$.
Moreover, it is said to be non-degenerate, if for any $f\in H^\bullet$ and any $h\in H$,
\center{$\langle f,H\rangle=0$ implies $f=0$, and $\langle H^\bullet,h\rangle=0$ implies $h=0$.}
\end{definition}

Consider one of the infinite-dimensional Taft algebras (\cite[Example 2.7]{LWZ07}), denoted by $T_\infty(2,1,-1)$. Suppose $T_\infty(2,1,-1)^\bullet$ is chosen as the link-indecomposable component of the finite dual $T_\infty(2,1,-1)^\circ$ containing the unit element. We would show that the evaluation $\langle,\rangle:T_\infty(2,1,-1)^\bullet\otimes T_\infty(2,1,-1)\rightarrow\k$ is a non-degenerate Hopf pairing.

Let us recall the structure of $T_\infty(2,1,-1)$ and $T_\infty(2,1,-1)^\circ$. We remark that the finite dual of infinite-dimensional Taft algebra $T_\infty(n,v,\xi)$ are once determined in \cite[Lemma 6.9]{Jah15} and \cite[Corollary 4.4.6(III)]{Cou19} (see also \cite[Proposition 7.2]{BCJ21}). Here we introduce the structure of $T_\infty(2,1,-1)^\circ$ stated in \cite[Section 3]{LL??2} with generators and relations:

\begin{example}
\begin{itemize}
\item[(1)]
As an algebra, $T_\infty(2,1,-1)$  is generated by $g$ and $x$ with relations:
\begin{align*}
g^2=1,\;\;xg=-gx.
\end{align*}
Then $T_\infty(2,1,-1)$ becomes a Hopf algebra with comultiplication, counit and antipode given by
\begin{align*}
& \Delta(g)=g\otimes g,\;\;\Delta(x)=1\otimes x+x\otimes g,\;\;
  \varepsilon(g)=1,\;\;\varepsilon(x)=0,  \\
& S(g)=g,\;\;S(x)=gx.
\end{align*}
Moreover, $T_\infty(2,1,-1)$ has a linear basis $\{g^jx^l\mid 0\leq j\leq 1,\;l\in\mathbb{N}\}$.

\item[(2)]
As an algebra, $T_\infty(2,1,-1)^\circ$ is generated by $\psi_\la\;(\la\in\k),\;\omega,\;E_2,\;E_1$ with relations
\begin{align*}
& \psi_{\la_1}\psi_{\la_2}=\psi_{\la_1+\la_2},\;\;\psi_0=1,\;\;
\om^2=1,\;\;E_1^2=0,  \\
& \om\psi_\la=\psi_\la\om,\;\;E_2\om=\om E_2,\;\;E_1\om=-\om E_1,  \\
& E_2\psi_\la=\psi_\la E_2,\;\;E_1\psi_\la=\psi_\la E_1,\;\;E_1E_2=E_2E_1
\end{align*}
for all $\la,\la_1,\la_2\in\k$. The coalgebra structure and antipode are given by:
\begin{align*}
& \Delta(\om)=\om\otimes\om,\;\;\Delta(E_1)=1\otimes E_1+E_1\otimes\om,\;\;\\
& \Delta(E_2)=1\otimes E_2 + E_1\otimes\om E_1+ E_2\otimes 1,  \\
& \Delta(\psi_\la)=(\psi_\la\otimes\psi_\la)(1\otimes 1+\la E_1\otimes\om E_1),  \\
& \varepsilon(\om)=\varepsilon(\psi_\la)=1,\;\;\varepsilon(E_1)=\varepsilon(E_2)=0,  \\
& S(\om)=\om,\;\;S(E_1)=\om E_1,\;\;S(E_2)=-E_2,\;\;
  S(\psi_\la)=\psi_{-\la},
\end{align*}
for $\la\in\k$. Note that $\{\psi_\la\om^j E_2^s E_1^l\mid \la\in\k,\;0\leq j,l\leq 1,\;s\in\mathbb{N}\}$ is a linear basis.
\end{itemize}
\end{example}

\begin{lemma}(\cite[Section 6]{LL??2})
$T_\infty(2,1,-1)^\circ$ has a Hopf subalgebra
$$T_\infty(2,1,-1)^\bullet:=\k\{\om^j E_2^s E_1^l\mid 0\leq j,l\leq 1,\;s\in\mathbb{N}\},$$
such that the evaluation $\langle,\rangle:T_\infty(n,v,\xi)^\bullet\otimes T_\infty(n,v,\xi)\rightarrow\k$ is a non-degenerate Hopf pairing.
\end{lemma}

At final, a similar process as Subsection \ref{subsection:D} follows the link-decomposition of $T_\infty(2,1,-1)^\circ$, by which we could identify the Hopf subalgebra $T_\infty(2,1,-1)^\bullet$ with a link-indecomposable component:

\begin{proposition}
The Hopf subalgebra $T_\infty(2,1,-1)^\bullet$ is exactly the link-indecomposable component of $T_\infty(2,1,-1)^\circ$ containing the unit element $1$.
\end{proposition}

\begin{proof}
Denote the Hopf algebra $T_\infty(2,1,-1)^\circ$ simply by $H$. We claim that
\begin{equation}\label{eqn:Tcircdecomp}
H=H_{(1)}\oplus\left(\bigoplus\limits_{\la\in\k^\ast}H_{(C_\la)}\right),
\end{equation}
where
\begin{eqnarray*}
H_{(1)} &=& \k\{\om^j E_2^s E_1^l\mid 0\leq j,l\leq 1,\;s\in\mathbb{N}\}=T_\infty(2,1,-1)^\bullet,  \\
H_{(C_\la)} &=& \k\{\psi_\la\om^j E_2^s E_1^l\mid 0\leq j,l\leq 1,\;s\in\mathbb{N}\}.
\end{eqnarray*}

In details, evidently the set of simple subcoalgebras $\mathcal{S}$ contains
$$\{\k1,\;\k\om,\;C_\la \mid \la\in\k^\ast\},$$
where $C_\la$ has a basic multiplicative matrix
$\C_\la:=\left(\begin{array}{cc}
    \psi_\la & \la\psi_\la E_1  \\  \psi_\la\om E_1 & \psi_\la\om
 \end{array}\right)$
for each $\la\in\k^\ast$, and hence $\om C=C\om=C$.

One could find that
$$\mathcal{E}:=\left(\begin{array}{ccc}
    1 & E_1 & E_2  \\  0 & \om & \om E_1 \\ 0 & 0 & 1
 \end{array}\right)$$
is a multiplicative matrix. Clearly $\k1$ and $\k\om$ are linked. For any $0\leq j,l\leq 1,\;s\in\mathbb{N}$, the element $\om^j E_2^s E_1^l$ is an entry (with some non-zero scalar) of the multiplicative matrix $\mathcal{E}^{\odot s}$. Thus
$$\k\{\om^j E_2^s E_1^l\mid 0\leq j,l\leq 1,\;s\in\mathbb{N}\}\subseteq H_{(1)}.$$

On the other hand, for any $\la\in\k^\ast$ and $0\leq j,l\leq 1,\;s\in\mathbb{N}$, the element $\psi_\la\om^j E_2^s E_1^l$ is an entry (with some scalar) of the multiplicative matrix $\mathcal{E}^{\odot s}\odot\C_\la$, whose diagonal is made up with basic multiplicative matrices of $C_\la$. Thus
$$\k\{\psi_\la\om^j E_2^s E_1^l\mid 0\leq j,l\leq 1,\;s\in\mathbb{N}\}\subseteq H_{(C_\la)}.$$

It could be concluded that
$\mathcal{S}=\{\k1,\;\k\om,\;C_\la \mid \la\in\k^\ast\}$ and (\ref{eqn:Tcircdecomp}) holds.
\end{proof}

\begin{remark}
When $H$ is pointed, it is stated in \cite[Theorem 3.2]{Mon93} that $H_{(1)}$ is always a normal Hopf subalgebra. As for the example $H=T_\infty(2,1,-1)^\circ$ in this subsection, one could verify that $H_{(1)}$ is also normal as a Hopf subalgebra, according the equations such as
$$\psi_\la \om^j E_2^s E_1^l\psi_{-\la}=\psi_\la\psi_{-\la} \om^j E_2^s E_1^l=\om^j E_2^s E_1^l.$$
Some other examples might also be verified.
\end{remark}

\section*{Acknowledgement}

The author is extremely grateful for the referee's detailed comments, as well as his/her significant argument and techniques achieving Section \ref{section4}, which make this paper greatly more complete and valuable.
The author would also like to thank Professor Gongxiang Liu and Professor Shenglin Zhu for useful discussions, especially on further possible results relevant to the paper.

\end{document}